\documentclass[11pt,reqno,oneside]{amsproc}

\usepackage{mathtools}
\usepackage{esint}
\usepackage{amsfonts}
\usepackage{dsfont}          
\usepackage{amsmath, amsthm, amssymb, enumerate}
\usepackage{cancel}

\usepackage{color}  
\usepackage{marginnote}
\usepackage[colorlinks=true, pdfstartview=FitV, linkcolor=blue, citecolor=blue, urlcolor=blue]{hyperref}

\chardef\forshowkeys=0
\chardef\refcheck=0
\chardef\showllabel=0
\chardef\sketches=0
\chardef\showcolors=0
\ifnum\forshowkeys=1

  \usepackage[notref,notcite,color]{showkeys}
\fi

\ifnum\showllabel=1
\def\llabel#1{\marginnote{\color{colorcccc}\rm\small(#1)}[-0.0cm]\notag}
   \def\llabel{\label}
\else
\def\llabel#1{\notag}
\fi

\ifnum\refcheck=1
\usepackage{refcheck}
\fi

\newtheorem{Theorem}{Theorem}[section]

\newtheorem{Lemma}[Theorem]{Lemma}

\theoremstyle{definition}

\def\ccc{C_m}

\def\constanta{A_m}
\def\constantn{N_m}

\def\bhh{h}
\def\bnn{n}
\def\bdd{d}

\newcommand{\diam}{\operatorname{diam}}
\newcommand{\mywidth}{\operatorname{width}}

\def\inon#1{\hbox{\ \ \ \ \ \ \ }\hbox{#1}}                

\def\inin#1{\inon{in~$#1$}}

\def\RR{{\mathbb R}}

\def\comma{ {\rm ,\quad{}} }            

\def\dive{\mathop{\rm div}\nolimits}    
\def\indeq{\qquad{}}                     

\def\RR{\mathbb R}

\def\eps{\epsilon}

\def\tilde{\widetilde}

\def\div{\mathop{\rm div}\nolimits}

\def\indeq{\quad{}}

\ifnum\showcolors=1
\definecolor{colorcccc}{rgb}{0.7,0.7,0.7}

\def\colb{\color{black}}
\definecolor{colorpppp}{rgb}{0.6,0.0,0.1}
\definecolor{colorgggg}{rgb}{.0,0.4,0.0}
\definecolor{colorhhhh}{rgb}{0,0.6,0.2}
\definecolor{colorgray}{rgb}{0.8,0.8,0.8}

\definecolor{coloroftheorems}{rgb}{0.45,0.0,0.0}
\definecolor{colorigor}{rgb}{1, 0.2, 0.8}
\definecolor{amethyst}{rgb}{0.6, 0.4, 0.8}

\def\cole{\color{coloroftheorems}}

\definecolor{colororange}{rgb}{0.8,0.2,0}
\definecolor{colorpurple}{rgb}{0.6,0.0,0.6}

\else   
\definecolor{colorcccc}{rgb}{0,0,0}

\def\colb{\color{black}}
\definecolor{colorpppp}{rgb}{0,0,0}
\definecolor{colorgggg}{rgb}{0,0,0}
\definecolor{colorhhhh}{rgb}{0,0,0}
\definecolor{colorgray}{rgb}{0,0,0}

\definecolor{coloroftheorems}{rgb}{0,0,0}
\definecolor{colorigor}{rgb}{0,0,0}
\definecolor{amethyst}{rgb}{0,0,0}

\def\cole{\color{coloroftheorems}}

\definecolor{colororange}{rgb}{0.8,0.2,0}
\definecolor{colorpurple}{rgb}{0.6,0.0,0.6}

\fi

\def\bega{\begin{aligned}}
\def\enda{\end{aligned}}

\def\bcase{\begin{cases}}
\def\ecase{\end{cases}}

\def\bmx{\begin{bmatrix}}
\def\emx{\end{bmatrix}}

\def\NN{\mathbb{N}}

\numberwithin{equation}{section}

\begin{document}

\baselineskip=12.6pt

$\,$
\vskip1.2truecm
\title[Hausdorff measure of Nodal set]{Nodal set for the Schr\"odinger equation under a\\ local growth condition}

\author[I.~Kukavica]{Igor Kukavica}
\address{Department of Mathematics, University of Southern California, Los Angeles, CA 90089}
\email{kukavica@usc.edu}

\author[L.~Li]{Linfeng Li}
\address{Department of Mathematics, University of California Los Angeles, Los Angeles, CA 90095}
\email{lli265@math.ucla.edu}

\begin{abstract}
We address the upper bound on the size of the nodal set
for a solution $w$ of
the Schr\"odinger equation
$\Delta w= W\cdot \nabla w+V  w$ in an open set in $\mathbb{R}^n$,
 where the coefficients belong to certain Sobolev spaces.
 Assuming a local doubling condition for the solution $w$, we establish an upper bound on the $(n-1)$-dimensional Hausdorff measure of the nodal set,
with the bound depending algebraically on the Sobolev norms of $W$ and~$V$.
	\hfill \today
\end{abstract}

\maketitle

\date{}

\section{Introduction}
In recent years, quantitative bounds on the size of nodal sets for solutions of elliptic and parabolic equations have attracted considerable attention, both for their geometric interest and for applications to inverse problems and control theory. 
In this paper, we focus on the Schr\"odinger equation
\begin{align}
	\Delta w
	= 
	W\cdot\nabla w 
	+
	V w 
	\label{EQ01}
\end{align}
in $\Omega\subset \mathbb{R}^n$, where $W$ and $V$ are given.
Under the assumption that a solution $w$ satisfies a natural $L^2$ growth condition, we derive an algebraic upper bound on the $(n-1)$-dimensional Hausdorff measure of the nodal set.
Notably, our estimate depends polynomially on the Sobolev norms of the
coefficients $W$ and $V$,
extending the previous work on eigenfunctions of the Laplacian
to this more general situation.

One of the fundamental questions in the study of zero sets is Yau's
conjecture (see~\cite{Y}), which asserts that for an eigenfunction $w$ of the Laplace-Beltrami operator on an $n$-dimensional compact Riemannian manifold $(M,g)$ with eigenvalue $\lambda$, the size of its nodal set satisfies
\begin{align}
	C_1 \lambda^{1/2}
	\leq
	\mathcal{H}^{n-1}
	(\{x\in M: w(x)=0\})
	\leq
	C_2 \lambda^{1/2},
	\llabel{EQ02}
\end{align}
for constants $C_1,C_2>0$ depending only on $(M,g)$.
Yau's conjecture was proven by Donnelly and Fefferman in \cite{DF1} under the real-analyticity assumption on the manifold.
For smooth manifolds, Hardt and Simon (see~\cite{HS}) proved an exponential upper bound $\mathcal{H}^{n-1} (\{x\in M: w(x)=0\}) \leq C \lambda^{C\sqrt{\lambda}}$.
Recently, Logunov and Malinnikova made a breakthrough towards Yau's conjecture for smooth manifolds~\cite{Lo1, Lo2,LoM1}.
In \cite{LoM1}, Logunov and Malinnikova proved $\mathcal{H}^1 (\{x\in M: w(x)=0\}) \leq C \lambda^{3/4-\eps}$ in dimension two, which slightly improves the upper bound $C\lambda^{3/4}$ by Donnelly and Fefferman \cite{DF2} and Dong \cite{Do}.
For the upper bounds in higher dimensions $n\geq 3$, Logunov in \cite{Lo1} obtained a polynomial upper bound
$\mathcal{H}^{n-1} (\{x\in M: w(x)=0\}) \leq C \lambda^\alpha$, where $\alpha>1/2$ depends only on the dimension.
The proof employs an almost monotonicity property for harmonic functions on Riemannian manifolds, a quantitative propagation-of-smallness scheme, and a combinatorial analysis of the doubling index that controls the number of high-vanishing-order cubes.
In \cite{Lo2}, Logunov proved the Nadirashvili's conjecture and showed the lower bound in Yau's conjecture for a compact $C^\infty$-smooth Riemannian manifold without boundary.
See~\cite{LoM2} for a survey of Yau's conjecture and  \cite{LZ, KZZ, LMNN, LS, LTY, NV, Z2, ZZ} for related results on the nodal sets of elliptic equations in various settings.
We emphasize that, after harmonic lifting, a Laplace eigenfunction can be regarded as a harmonic function on a suitable product manifold, which allows application of an almost monotonicity property for harmonic function (see \cite{Lo1} for instance). 
In contrast, no analogous reduction is available for the more general Schr\"odinger-type equation \eqref{EQ01} with variable lower-order terms.

The study of nodal sets is closely connected to the unique continuation property for partial differential equations, which asserts that a nontrivial solution cannot vanish on any open set.
In particular, the quantitative unique continuation results for the Schr\"odinger operator $-\Delta + V$ have been extensively studied using two techniques.
The first relies on Carleman estimates, which were pioneered by Carleman (see~\cite{Car}) and Aronszajn (see~\cite{Ar});
see also \cite{ABG, AKS, DZ, JK, SS} for related results.
The second is the frequency approach, which was introduced by
Almgren (see \cite{Al}) for harmonic functions and extended
in \cite{GL} by Garofalo-Lin
to elliptic operators with sufficiently smooth coefficients.
The frequency approach leads in some cases to a sharp bound
on the order of vanishing (see~\cite{Ku2} for instance);
see also \cite{BG, Ku1, Z1} for related results.
More recently, Davey in \cite{Da} obtained quantitative strong unique continuation estimates for the generalized Schr\"odinger equation by employing a modified frequency function approach, providing an alternative to earlier proofs based on Carleman estimates.
We emphasize that the frequency approach seems to lead to \emph{almost monotonicity} of the frequency, which is one of the important devices needed in the paper~\cite{Lo1} by Logunov.

The goal of this paper is to provide a polynomial upper bound on the Hausdorff measure of the nodal set for the Schr\"odinger equation \eqref{EQ01}, under a local doubling condition concerning the growth of the solution.
After rescaling, we introduce a refined $L^2$-based frequency function and prove its almost-monotonicity, together with a corresponding control of the $L^2$-based doubling index.
To translate doubling estimates into nodal set bounds,
we draw ideas from Logunov (see~\cite{Lo1}), including dissecting the cube, small-scale Cauchy uniqueness arguments, and a combinatorial covering scheme to prevent high-vanishing-order regions from clustering.
We emphasize that we use the frequency function associated to the full ball~\cite{KN}, leading to some simplifications in the arguments (see Lemmas~ \ref{L00}--\ref{L04} for instance).
Finally, building on the modified frequency function, we extend the classical three-ball argument to a four-ball inequality which, together with our growth hypothesis, yields a uniform upper bound on the doubling index.

The paper is organized as follows.
In Section~\ref{sec02}, we state the main result, Theorem~\ref{T01}, asserting a polynomial upper bound on the size of the nodal set.
In Section~\ref{sec03}, we introduce the $L^2$-based frequency function and the doubling index,
and 
establish the almost monotonicity property for the latter.
In Section~\ref{sec04}, we derive an upper bound on the number of dyadic cubes with large doubling index, and in Section~\ref{sec05} we combine the bound with our global doubling index estimate to complete the proof of the main result.

\section{Main result}
\label{sec02}
Consider a solution $w$ of the Schr\"odinger equation
\begin{align}
\Delta w
= 
W\cdot \nabla w
+
V  w
\label{EQ03}
\end{align}
in a neighborhood of $\bar B_2\subset \mathbb{R}^n$, where $n\geq 2$.
We assume the solution $w$ satisfies a local doubling inequality
\begin{align}
	\Vert w\Vert_{L^2 (B_{2})}
	\leq
	e^\kappa
	\Vert w\Vert_{L^2 (B_{1})}
	,
	\label{EQ04}
\end{align}
where $\kappa\geq 1$ is a constant.
Set
$M=
\Vert V\Vert_{W^{1,\infty} (B_2)}
+
\Vert W\Vert_{W^{1,\infty} (B_2)}$+1.

The following theorem provides an algebraic upper bound for the Hausdorff measure of the nodal set in terms of the coefficients.

\cole
\begin{Theorem}
\label{T01}
Let $n\geq 2$.
Suppose $w$ is a solution of \eqref{EQ03} in a neighborhood of
$\bar B_2\subset \RR^n$ and satisfies~\eqref{EQ04}.
Then we have
\begin{align}
	\mathcal{H}^{n-1}
	(\{w=0\}\cap B_1)
	\leq
	C 
	\left(M^2
	+
	\kappa
	\right)^C,
	\llabel{EQ05}
\end{align}
\colb
where $C>0$ is a constant.
\end{Theorem}
\colb

Throughout the paper, we use $C\geq 1$ to denote a sufficiently large constant, which may depend on $n$ without mention; the value of $C$ may vary from line to line.
We use $B_r (x)$ to denote a ball of radius $r$ centered at the point $x$, abbreviated by $B_r$ when the center is~$0$.

\section{Almost monotonicity of the doubling index}
\label{sec03}
We consider a solution $u$ to
\begin{align}
\Delta u
=
W\cdot \nabla u
+
V u
\label{EQ06}
\end{align}
in a domain $\Omega\subset \RR^n$, 
where the coefficients $W$ and $V$ satisfy 
\begin{align}
	\Vert V\Vert_{W^{1,\infty} (\Omega)} + 
	\Vert W\Vert_{W^{1,\infty} (\Omega)}
	\leq 1.
	\label{EQ07}
\end{align}
In Sections~\ref{sec03}--\ref{sec05}, the solution $u$ is assumed to satisfy \eqref{EQ06}--\eqref{EQ07}, unless stated otherwise.

\subsection{The frequency function}
For $x\in \Omega$ and $r>0$, we define
\begin{equation}
H(x,r) 
= 
\int_{B_r (x)} u^2 (y)
\,dy
\llabel{EQ08}
\end{equation}
and
\begin{equation}
I(x,r)
=
\int_{B_r (x)} 
|\nabla u (y)|^2 (r^2-|y-x|^2)
\,dy.
\llabel{EQ09}
\end{equation}
We define the frequency function by
\begin{align}
\begin{split}
	F(x,r)
	=
	\frac{I(x,r)}{2H(x,r)}
	.
\end{split}
\llabel{EQ10}
\end{align}

The following lemma establishes the almost monotonicity of the frequency function.
\cole
\begin{Lemma}
\label{Lmonotonenew}
There exist constants $C,\ccc>0$ such that
	\begin{equation}
	\left|
	\frac{H'(x,r)}{H(x,r)}
	- 
	\frac{n}{r}
	- 
	\frac{2F (x,r)}{r}
	\right|
	\leq C (r+1)
	\label{EQ11}
\end{equation}
and
\begin{align}
	\frac{d}{dr}
	\left(	\left(
	F (x,r)
	+
	\ccc r^2 
	+ 
	\ccc r^{4}
	\right)
	e^{2r^2 +4r}\right)
	\geq 0,
	\label{EQ12}
\end{align}
for all $x\in \Omega$.
\end{Lemma}
\colb

\begin{proof}
By translation, we may assume that $x=0$.
Direct computation gives
\begin{align}
\begin{split}
	H'(r)
	&=
	\int_{\partial B_r (0)} u^2\,d\sigma
	=
	\frac{1}{r} 
	\int_{B_r (0)} \div(u^2 y)
	= 
	\frac{n}{r} H(r)
	+
	\frac{2}{r}
	\int_{B_r (0)} y_j u\partial_{j} u
	,
\end{split}
\label{EQ13}
\end{align}
which leads to
\begin{align}
	\begin{split}
		H'(r)
		=
		\frac{n}{r} H(r)
		- 
		\frac1r
		\int_{B_r (0)} u\partial_{j}u \partial_{j}(r^2-|y|^2)
		.
	\end{split}
	\llabel{EQ14}
\end{align}
Here and below, all unindicated integrals are taken over $B_r (0)$ with respect to~$y$. 
We integrate by parts and use \eqref{EQ06} to obtain
\begin{align}
\begin{split}
	H'(r)
	&=
	\frac{n}{r} H(r)
	+
	\frac1r
	\int |\nabla u|^2 (r^2-|y|^2)
	+ 
	\frac1r
	\int u \Delta u
	(r^2-|y|^2)
	\\&
	=
	\frac{n}{r} H(r)
	+
	\frac{1}{r}
	I(r)
	+ 
			\frac{1}{r}
	\int V u^2
	(r^2-|y|^2)
	+
			\frac{1}{r}
	\int u W \cdot\nabla u
	(r^2 - |y|^2)
	.
\end{split}
\label{EQ15}
\end{align}
For the last term on the far right-hand side of \eqref{EQ15}, we integrate by parts to obtain
\begin{align}
\int
u W \cdot \nabla u (r^2 - |y|^2)
=
-\frac{1}{2}
\int \dive W u^2
(r^2 - |y|^2)
+
\int W_k y_k u^2.
\label{EQ16}
\end{align}
From \eqref{EQ15} and \eqref{EQ16}, it follows that
\begin{align}
	\left|H' -\frac{n}{r} H - \frac{1}{r} I 
	\right|
	\leq
	C(r+1)H,
	\llabel{EQ17}
\end{align}
which completes the proof of~\eqref{EQ11}.
Note that
\begin{align}
	I(r)= r^{n+2}\int_{B_1} \partial_{j} u(ry)\partial_{j}u(ry)(1-|y|^2)
	.
	\label{EQ18}
\end{align}
Differentiating \eqref{EQ18} gives
\begin{align}
	\begin{split}
		I'
		&=
		(n+2) r^{n+1}
		\int_{B_1}
		\partial_{j} u(ry)\partial_{j}u(ry)(1-|y|^2)
		+
		2r^{n+2}
		\int_{B_1}
		y_k \partial_{jk}u(ry) \partial_{j}u(ry) (1-|y|^2)
		\\&
		=
		\frac{n+2}{r}
		I
		+ 
		\frac{2}{r}
		\int y_k \partial_{jk} u \partial_{j}u (r^2-|y|^2)
		\\&
		=
		\frac{n+2}{r}
		I
		- 
		\frac2r
		\int \delta_{jk}
		\partial_{j}u\partial_{k}u (r^2-|y|^2)
		- \frac2r
		\int
		y_k\partial_{k} u \Delta u (r^2-|y|^2)
		+ 
		\frac4r
		\int y_k \partial_{k} u y_j \partial_{j} u
		\\&
		=
		\frac{n}{r}
		I
		+
		\frac4r
		\int y_k \partial_{k} u y_j \partial_{j} u
		- \frac2r
		\int
		y_k V u\partial_{k}u  (r^2-|y|^2)
			- \frac2r
		\int
		y_k  W\cdot \nabla u
		\partial_{k}u  (r^2-|y|^2)
		,
	\end{split}
	\llabel{EQ19}
\end{align}
where we used \eqref{EQ06} in the last step.
Hence, we obtain
\begin{align}
	\begin{split}
		I'H 
		&=
		\frac{n}{r} I H
		+
		\left(
		\frac4r
		\int y_k \partial_{k} u y_j \partial_{j} u
		\right)
		\left(
		\int u^2
		\right)
		- 
		\frac{2	H}{r}
		\int
		y_k V u\partial_{k}u  (r^2-|y|^2)
				\\&\quad\quad
		- 
		\frac{2	H}{r}
		\int
		y_k W\cdot \nabla u\partial_{k}u  (r^2-|y|^2).
	\end{split}
	\label{EQ20}
\end{align}
Using \eqref{EQ13} and the identity
\begin{equation}
	I(r)
	=
	2\int y_j u\partial_{j}u
	-
	\int V u^2 (r^2-|y|^2)
	-
	\int u W\cdot \nabla u (r^2-|y|^2)
,
	\label{EQ21}
\end{equation}
we obtain
\begin{align}
	\begin{split}
		I H'
		&
		=
		\frac{n}{r}   
		I H
		+ 
		\left(
		\frac{2}{r}
		\int y_j u\partial_{j} u
		\right)
		\left(
		2
		\int y_j u\partial_{j} u
		\right)
		-
		\left(\frac{2}{r}
		\int y_j u\partial_{j} u
		\right)
		\left(
		\int V u^2(r^2-|y|^2)
		\right)						\\&\quad\quad
		-
		\left(
		\frac{2}{r}
		\int y_j u\partial_{j} u
		\right)
		\left(
		\int
		u W\cdot \nabla u 
		(r^2 -|y|^2)
		\right)
		.
	\end{split}
	\label{EQ22}
\end{align}
From \eqref{EQ20} and \eqref{EQ22} and the Cauchy-Schwarz inequality, it follows that
\begin{align}
	\begin{split}
		F'(r)
		&=
		\frac{I' H  - I H' }{2 H^2 }
		\geq
		-
		\frac{1}{rH}
		\int
		y_k V u\partial_{k}u  (r^2-|y|^2)
		+
		\left(	\frac{1}{rH^2}
		\int y_j u \partial_{j} u\right)
		\left(	\int V u^2(r^2-|y|^2)\right)
		\\&\quad\quad
		-
		\frac{1}{rH}
		\int y_k W\cdot \nabla u \partial_k u (r^2 -|y|^2)
		+
		\left(	\frac{1}{rH^2}
		\int y_j u \partial_{j} u\right)
		\left(
		\int
		u W \cdot \nabla u (r^2 - |y|^2)
		\right)
		\\&
		:=
		\mathcal{I}_1
		+
		\mathcal{I}_2
		+	
		\mathcal{I}_3
		+	
		\mathcal{I}_4
		.
	\end{split}
	\label{EQ23}
\end{align}
For the term $\mathcal{I}_1$, we use the Cauchy-Schwarz inequality to get
\begin{align}
|\mathcal{I}_1|
\leq
\frac{r}{H} \int |u| |\nabla u|
(r^2 -|y|^2)^{1/2}
\leq
\frac{r}{H} 
\left(\int u^2\right)^{1/2}
\left(
\int
|\nabla u|^2 (r^2 -|y|^2)
\right)^{1/2}
\leq
\sqrt{2} rF^{1/2}
,
	\label{EQ24}	
\end{align}
while for $\mathcal{I}_2$,
we appeal to \eqref{EQ21} and \eqref{EQ24}, yielding
\begin{align}
\begin{split}
	|\mathcal{I}_2|
&\leq
\frac{I}{2rH^2}
	\left(	\int |V| u^2(r^2-|y|^2)
	\right)
	+
	\frac{1}{2rH^2}
	\left(\int
	|V| u^2 (r^2-|y|^2)\right)
	\left(	\int |V| u^2(r^2-|y|^2)
	\right)
	\\&\quad\quad
	+
	\frac{1}{2rH^2}
	\left(\int
	|u W\cdot \nabla u| (r^2-|y|^2)\right)
	\left(	\int |V| u^2(r^2-|y|^2)
	\right)
	\\&
	\leq
	rF
	+
	\frac{r^3}{2}
	+
	\frac{r^2 F^{1/2}}{\sqrt{2} }.	\label{EQ25}
\end{split}
\end{align}
The term $\mathcal{I}_3$ can be estimated as
\begin{align}
|\mathcal{I}_3|
\leq
\frac{1}{H}
\int |\nabla u|^2 
(r^2 -|y|^2)
=
2F,
\label{EQ26}
\end{align}
and for $\mathcal{I}_4$ we use \eqref{EQ16} and proceed analogously to~\eqref{EQ25} to get
\begin{align}
	\begin{split}
	|\mathcal{I}_4|
	&\leq	
	\frac{I}{2rH^2}
	\Big|
	\int
	u W \cdot \nabla u (r^2 - |y|^2)
	\Big|
	+
	\frac{1}{2rH^2}
	\left(\int
	|V| u^2 (r^2-|y|^2)\right)
		\left(
	\int
	|u W \cdot \nabla u |
	(r^2 - |y|^2)
	\right)
	\\&\quad\quad
	+
	\frac{1}{2rH^2}
	\left(\int
	|u W\cdot \nabla u| (r^2-|y|^2)\right)
	\left(
	\int
	|u W \cdot \nabla u| (r^2 - |y|^2)
	\right)
	\\&
	\leq
	\frac{I}{2rH^2}
	\left(
	\frac{1}{2} r^2 H +rH
	\right)
	+
	\frac{r}{2H} r I^{1/2} H^{1/2}
	+
	\frac{I^{1/2}}{2H^{3/2
	}} rI^{1/2} H^{1/2}
	\\&
	\leq 
	F
	+
	\frac{3r}{2} F
	+
	\frac{r^2 F^{1/2}}{\sqrt{2}}.
\label{EQ27}
\end{split}
\end{align}
Combining \eqref{EQ23}--\eqref{EQ27}, we arrive at
\begin{align}
\begin{split}
	F'(r)
	\geq
	-\sqrt{2} rF^{1/2}
	-\frac{r^3}{2}
	-\sqrt{2} r^2 F^{1/2}
	-\frac{5r}{2} F-3F 
	,
\end{split}
\llabel{EQ28}
\end{align}
which leads to
\begin{equation}
\left(
\left(
F (r)
+
\ccc r^2 
+ 
\ccc r^{4}
\right)
e^{2r^2 +4r}
\right)'
\geq 0
,
\llabel{EQ29}
\end{equation}
for some constant $\ccc\geq 1$,
and \eqref{EQ12} follows.
\end{proof}

The following lemma provides an estimate for the growth of $H (x,r)$ in terms of the frequency function.
\cole
\begin{Lemma}
\label{L00}
Let $\eps\in (0,1/2]$.
For $0<r_1< r_2\leq -\log (1-\eps)/6\ccc$, 
where $C_m\geq 1$ is the constant from Lemma~\ref{Lmonotonenew},
we have
\begin{align}
F (x,r_2)
\geq 
(1-\epsilon) 
F(x,r_1)
-\eps
\label{EQ30}
\end{align}
and
\begin{align}
\begin{split}
&\left(
1
+
2(1-\eps) F (x,r_1) 
\right)
\log
\left(
\frac{r_2}{r_1}
\right)
-C
\leq
\log
\left(
\frac{ H (x,r_2)}{ H (x,r_1)}
\right)
\\&\quad\quad
\leq
	\left(	
	C
	+ 
	\frac{2}{1-\eps}F (x,r_2)
	\right)
	\log
	\left(
	\frac{r_2}{r_1}
	\right)
	+C,
	\label{EQ31}
\end{split}
\end{align}
for any $x\in \Omega$.
\end{Lemma}
\colb
We emphasize that the constant $C\geq1$ is independent of~$\eps$.

\begin{proof}
Let $x\in\Omega$.
For $0<r_1 <r_2<-\log (1-\eps)/6\ccc$, Lemma~\ref{Lmonotonenew} implies that
\begin{align}
	\left(F(x,r_2) 
	+ 
	\ccc r_2^2 
	+\ccc r_2^4
	\right) e^{ 2r_2^2 +4r_2}
	\geq
	\left(F(x,r_1) 
	+ 
	\ccc r_1^2 
	+\ccc r_1^4
	\right) e^{ 2r_1^2 +4r_1}
	> F(x,r_1),
	\llabel{EQ32}
\end{align}
which leads to~\eqref{EQ30}.
Next, we integrate \eqref{EQ11} on both sides from $r_1$ to $r_2$ and appeal to \eqref{EQ30}, yielding
\begin{align}
\begin{split}
	\log
	\left(	\frac{H(x,r_2)}{H(x,r_1)}\right)
	&\leq
	C\log
	\left(
	\frac{r_2}{r_1}
	\right)
	+
	\int_{r_1}^{r_2} \frac{2F(x,\rho)}{\rho}\,d\rho
	+
	C
	\\&
	\leq
	C \log \left(\frac{r_2}{r_1}\right)
	+
	\frac{2(\eps+ F(x,r_2))}{1-\eps}
	\log \left(
	\frac{r_2}{r_1}
	\right)
	+C
		\\&
	\leq
	\left(	C+ \frac{2}{1-\eps} F(x,r_2)
	\right)
	\log
	\left(
	\frac{r_2}{r_1}
	\right)
	+C
	\llabel{EQ33}
\end{split}
\end{align}
and
\begin{align}
\begin{split}
\log
\left(	\frac{H(x,r_2)}{H(x,r_1)}
\right)
&\geq
n
\log 
\left(
\frac{r_2}{r_1}
\right)
+
2((1-\eps) F(x,r_1)-\eps)
\log
\left(
\frac{r_2}{r_1}
\right)
-C
	\\&
\geq
\left(
1+
2(1-\eps)
F(x,r_1) 
\right)
\log
\left(
\frac{r_2}{r_1}
\right)-C.
\llabel{EQ34}
\end{split}
\end{align}
Combining the above two inequalities, we complete the proof of~\eqref{EQ31}.
\end{proof}

\subsection{The doubling index}
We define the doubling index as
\begin{align}
N_u (x,r)
=
\log_2 
\left(
\frac{H (x,2r)}{H (x,r)}
\right)
=
\log_2
\left(
\frac{\int_{B_{2r}(x)}  u^2}{\int_{B_r(x)}  u^2}
\right)
.
\llabel{EQ35}
\end{align}

The following lemma establishes the almost monotonicity of the doubling index.
\cole
\begin{Lemma}
\label{L02}
Let $\eps \in (0,1/2]$.
For $0<r<-\log (1-\eps)/12\ccc$, where $\ccc\geq 1$ is the constant from Lemma~\ref{Lmonotonenew},
we have
\begin{align}
2(1-\eps) F(x,r)
-
C
\leq
N (x,r)
\leq
\frac{2}{1-\eps} F (x,2r)
+C,
\label{EQ36}
\end{align}
for all $x\in \Omega$.
For $0<2r_1\leq r_2<-\log (1-\eps)/12\ccc$, we have
\begin{align}
N(x,r_2)
\geq 
(1-\eps)^3
N(x,r_1) 
-
C,
\label{EQ37}
\end{align}
for all $x\in \Omega$.
\end{Lemma}
\colb
We emphasize that the constant $C\geq1$ is independent of~$\eps$.

\begin{proof}
Let $x\in \Omega$.
For $r\in (0, -\log (1-\eps)/12\ccc)$, we appeal to Lemma~\ref{L00} to obtain
\begin{align}
\begin{split}
	2(1-\eps) F(x,r)
	-C
	\leq
	\log_2
	\left(
	\frac{H (x,2r)}{H (x,r)}
	\right)
	\leq
	\frac{2}{1-\eps}
	F(x,2r)+C
	\llabel{EQ38}
       ,
\end{split}
\end{align}
which gives~\eqref{EQ36}.
Let $0<2r_1\leq r_2<-\log (1-\eps)/12 \ccc$.
We infer from Lemma~\ref{L00} and \eqref{EQ36} that
\begin{align}
\begin{split}
	N(x,r_2)
	&\geq 
	2(1-\eps) F(x,r_2)
	-
	C
\geq 2(1-\eps)
((1-\eps) F(x,2r_1) -\eps)
-C
\\&
\geq
2(1-\eps)^2
\cdot
\frac{1-\eps}{2}
(N(x,r_1)-C)
-C
\\&
\geq 
(1-\eps)^3
N(x,r_1)
-C,
\llabel{EQ39}
\end{split}
\end{align}
completing the proof of~\eqref{EQ37}.
\end{proof}

The following lemma provides an estimate on the growth of $H(x,r)$ in terms of the doubling index.
\cole
\begin{Lemma}
\label{L04}
Let $\eps \in (0,1/2]$.
There exists a constant $C\geq1$ such that 
\begin{align}
	t^{(1-\epsilon)^2 N(x,r)-C}
	\leq
	\frac{H(x,tr)}{H(x,r)}
	\leq
	t^{(1-\epsilon)^{-2} N (x,tr) +C},
	\label{EQ40}
\end{align}
for
$r\in 
\left(0, -\log (1-\eps/2)/24\ccc
\right)$, $
t\in \left[2, -\log (1-\eps/2)/12 r \ccc\right]$,
and $x\in \Omega$, where $C_m\geq 1$ is the constant from Lemma~\ref{Lmonotonenew}.
Consequently, there exists a constant $N_0>0$, depending on $\eps$, such that if $N (x,r) \geq N_0$ for $r\in 
\left(0, -\log (1-\eps/2)/24\ccc
\right)$ and $x\in \Omega$, then
\begin{align}
t^{(1-\eps)^3 N (x,r)}
\leq
\frac{H(x,tr)}{H(x,r)}
\leq
t^{(1-\eps)^{-3} N (x,tr)}
\comma
2\leq  t \leq 	\frac{-\log (1-\eps/2)}{12 r \ccc}.
\label{EQ42}
\end{align}
\end{Lemma}
\colb

\begin{proof}
Let $x\in \Omega$,
$r\in \left(0, -\log (1-\eps/2)/24\ccc \right)$, and $t\in [2, -\log (1-\eps/2)/12rC_m]$.
By Lemmas~\ref{L00} and~\ref{L02}, we have
\begin{align}
\begin{split}
&
\log
\left(
\frac{H(x,tr)}{2^{N(x,r)} H(x,r)}
\right)
=
\log
\left(
\frac{H(x,tr)}{H(x,2r)}
\right)
\geq
\left(
1+
2(1-\eps) 
F(x,2r)
\right)
\log 
\left(
\frac{tr}{2r}
\right)
-C
\\&\quad\quad
\geq
\left(
1
+ 
2(1-\eps)
\cdot
\frac{1-\eps}{2}
\cdot
(N(x,r)-C)
\right)
\log \left(
\frac{t}{2}
\right)
-C
\\&\quad\quad
\geq
\left(
-C+(1-\eps)^2 N(x,r)
\right)
\log 
t
-
N(x,r)
\log 2,
\llabel{EQ43}
\end{split}
\end{align}
which leads to the lower bound in~\eqref{EQ40}.
Similarly, we have
\begin{align}
\begin{split}
	&
	\log
	\left(
	\frac{H(x,tr)}{2^{N(x,r)} H(x,r)}
	\right)
	=
	\log
	\left(\frac{H(x,tr)}{H(x,2r)}\right)
	\leq
	\left(C
	+ 
	\frac{2}{1-\eps}
	F(x, tr)
	\right)
	\log 
	\left(
	\frac{t}{2}
	\right)
	+C
	\\&\quad\quad
	\leq
	\left(
	C
	+ 
	\frac{2}{1-\eps}
	\cdot
	\frac{1}{2(1-\eps)}
	\cdot
	(N(x,tr)+C)
	\right)
	\log \left(\frac{t}{2}\right)
	+C
	\\&\quad\quad
	\leq
	\left(
	C+\frac{1}{(1-\eps)^2} N(x,tr)
	\right)
	\log t
	-
	\frac{1}{(1-\eps)^2}
	N(x,tr)
	\log 2.
	\llabel{EQ44}
\end{split}
\end{align}
Hence, 
\begin{align}
	\log 
	\left(
	\frac{H(x,tr)}{H(x,r)}
	\right)
	\leq
	N(x,r) \log 2
	-\frac{1}{(1-\eps)^2}
	N(x,tr) \log 2
	+\left(
	C+\frac{1}{(1-\eps)^2} N(x,tr)
	\right)
	\log t
        .
   \llabel{EQ45}
\end{align}
From Lemma~\ref{L02}, it follows that
\begin{align}
N(x,tr)\geq 
\left(1-\frac{\eps}{2}
\right)^3 
N(x,r)-C,
\label{EQ46}
\end{align}
from where
\begin{align}
	\log 
	\left(
	\frac{H(x,tr)}{H(x,r)}
	\right)
	\leq
	\left(
	C+\frac{1}{(1-\eps)^2} N(x,tr)
	\right)
	\log t
        ,
	\llabel{EQ47}
\end{align}
providing the upper bound in~\eqref{EQ40}.
Using \eqref{EQ40},
we conclude the proof of \eqref{EQ42}
by appealing to \eqref{EQ46} and taking $N_0>0$ sufficiently large, depending on~$\eps$.
\end{proof}

\section{Number of cubes with large doubling index}
\label{sec04}
\subsection{Simplex lemmas}
A simplex $S\subset \RR^n$ with vertices $x_1,x_2,\ldots,x_{n+1} \in\mathbb{R}^n$ is defined as the convex hull
\begin{align}
S=\left\{
\sum_{i=1}^{n+1} t_{i} x_{i}: t_1+t_2+\cdots +t_{n+1}=1,~ t_i\geq 0\right\}.
   \llabel{EQ48}
\end{align}
We denote the barycenter of the simplex by $x_0= \sum_{i=1}^{n+1} x_i / (n+1)$ and the relative width by $\omega(S):=\mywidth S /\diam S$, where $\mywidth S$ is the minimal distance between a pair of parallel hyperplanes such that $S$ is contained between them.

The following lemma gives a slight increment of the doubling index at the barycenter of a simplex from the doubling indices at its vertices.
\cole
\begin{Lemma}
\label{LSimplex}
Let $\alpha>0$.
There exist constants $ C_0, C_1,C_2, K_\alpha, N_\alpha>0$, depending on $\alpha$, such that if 
\begin{enumerate}
	\item $\omega (S) >\alpha$ and $\diam S<C_0$ and
\item there exist $r_j\in (0,K_\alpha\diam S/2]$, where $j=1,2,\ldots,n+1$, with
\begin{align}
	\min_{j=1,\ldots,n+1} N(x_j, r_j) 
	\geq 
	N_\alpha,
   \llabel{EQ49}
\end{align}
\end{enumerate}
then
\begin{align}
N (x_0, C_2 \diam S) 
>(1+C_1) \min_{j=1,\ldots,n+1} N(x_j, r_j).
\llabel{EQ50}
\end{align}
\end{Lemma}
\colb
This lemma is an analog of~\cite[Lemma~2.1]{Lo1}.
For completeness,
we provide the proof in Appendix~\ref{AppSimplex}.

Let $q\subset \mathbb{R}^n$ be a closed cube.
We denote the relative width of $E\subset q$ by
\begin{align}
	\tilde{\omega}_q (E) 
	= \frac{\mywidth{E}}{\diam{q}},
   \llabel{EQ51}
\end{align}
where $\mywidth E$ is the minimal distance between a pair of parallel hyperplanes such that $E$ is contained between them.

We recall an Euclidean geometry lemma from \cite[Theorem~5.1]{Lo1}.
\cole
\begin{Lemma}
[\cite{Lo1}]
\label{LG01}
Let $q\subset \mathbb{R}^n$ be a cube and $E\subset q$ with $\tilde{w}_q (E)>0$.
There exists a non-decreasing function $a= a(\tilde{\omega}_q (E))>0$ and a simplex $S$ with vertices $x_1,x_2,\ldots, x_{n+1} \in E$ such that
\begin{align}
	\omega (S)>a
	\quad \quad
	~~~\text{and}~~~\quad
	\diam S>a \cdot \diam q.
   \llabel{EQ52}
\end{align}
\end{Lemma}
\colb

\subsection{Cauchy uniqueness}
The following Cauchy uniqueness theorem appears in \cite[Lemma~4.3]{Lin} and is proven in greater generality in \cite[Theorem~1.7]{ARRV}.

\cole
\begin{Lemma}
[\cite{Lin}]
\label{cauchyunique}
Let $u$ be a solution of \eqref{EQ06}--\eqref{EQ07} in the half ball $B_1^{+} = \{x\in \mathbb{R}^n: |x|<1, x_n>0\}$, and set
$\Gamma=\{x\in\mathbb{R}^n: |x|<3/4, x_n=0\}$.
There exists a constant $\eps\in (0,1)$ such that if
\begin{align}
\int_{B_1^{+}} |u|^2 \leq 1
   \llabel{EQ53}
\quad\quad
and\quad\quad
	\sup_{\Gamma} |u|
+
\sup_{\Gamma} |\nabla u|
\leq 
\eps
\end{align}
then
\begin{align}
\sup_{B_{1/2}^+ } 
|u|
\leq
C\eps^\beta,
	\llabel{EQ54}
\end{align}
where $C\geq1$ and $\beta>0$ are constants.
\end{Lemma}
\colb

We state the setup in the next three lemmas.
Let $R\in (0,1]$ and $Q=[-R/2, R/2]^n\subset \Omega$ be a closed cube. We define the doubling index of $u$ in a cube $Q$ as
\begin{align}
	N_u (Q)
	=\sup_{x\in Q, r\in (0,\diam Q )}
	N_u (x,r)
	=
	\sup_{x\in Q, r\in (0,\diam Q )}
	\log_2 \left(
	\frac{\int_{B_{2r}(x) } u^2}{\int_{B_r (x)} u^2}
	\right).
	\llabel{EQ55}
\end{align}
Let $A\geq 3$ be an odd integer.
We divide $Q$ into $A^n$ subcubes $q$ of equal size and denote the subcubes that intersect the hyperplane $\{x_n=0\}$ by $q_{j,0}$, where $j=1,2,\ldots, A^{n-1}$.

The first lemma establishes a lower bound of the doubling index of a cube in terms of the doubling indices of the subcubes near the hyperplane $\{x_n=0\}$.
\cole
\begin{Lemma}
\label{L07}
Let $u$ be a solution of \eqref{EQ06}--\eqref{EQ07}. 
There exist an integer $A_0\geq 3$ and a
constant $N_0>0$ such that if
$A\geq A_0$ and
\begin{align}
	\min_{j=1,\ldots,A^{n-1}} N (q_{j,0}) 
	\geq 
	N_0,
	\llabel{EQ56}
\end{align}
then
\begin{align}
	N (Q)
	> 
	2 \min_{j=1,\ldots,A^{n-1}} N (q_{j,0}).
   \llabel{EQ57}
\end{align}
\end{Lemma}
\colb
We emphasize that $A_0$ and $N_0$ are independent of~$R$.

\begin{proof}[Proof of Lemma~\ref{L07}]
First we consider the case when $Q=[-1/2,1/2]^n$.
Let $K\geq 1$ be a constant to be determined below; in particular, we shall choose $K$ to depend on $\ccc$ only, where $\ccc\geq 1$ is the constant from Lemma~\ref{Lmonotonenew}.
Set
\begin{align}
M=H(0, 1/8K)=\int_{B (0,1/8K)} u^2,
\label{EQ58}
\end{align}
and let $A_0\geq 3$ and $N_0>0$ be large constants,
both
to be determined below.
Also, we write $\Lambda=B(0,1/16K) \cap \{x_n = 0\}$ and $N=	\min_{j=1,\ldots,A^{n-1}} N (q_{j,0})$. 
Let $x \in \Lambda$ such that
\begin{align}
|u(x)|^2+ |\nabla u(x)|^2
\geq 
\frac{1}{2}
\sup_{y\in\Lambda} (|u(y)|^2+ |\nabla u(y)|^2).
\label{EQ59}
\end{align} 
There exists $j\in \{1,2,\ldots,A^{n-1}\}$ such that
$x\in q_{j,0} \subset B(0,1/12K)$.
Since $N(q_{j,0})\geq N$,
there exists some $x_j \in q_{j,0}$ and $r_j \leq \diam q_{j,0}=\sqrt{n}/A$ such that $N(x_j, r_j) \geq N$. 
Since $B(x_j , 1/32K) \subset B(0, 1/8 K)$, we infer that
\begin{align}
H(x_j, 1/32 K)
\leq
H(0,1/8K)
= M.
\label{EQ60}
\end{align}
Using Lemma~\ref{L02} and taking $A_0,N_0>0$ sufficiently large, we obtain
\begin{align}
	N\left(x_j, \frac{4\sqrt{n}}{A}\right)
	\geq 
	\frac{99}{100} N (x_j,r_j) -C
	\geq
	\frac{9}{10}
	N (x_j,r_j),
\llabel{EQ61}
\end{align}
from where we use Lemma~\ref{L04} to write
\begin{align}
	\frac{H(x_j, 1/32K)}{H(x_j, 4\sqrt{n}/A)}
	\geq \left(\frac{A}{128K\sqrt{n}}\right)^{\frac{5}{9} N \left(x_j, 4\sqrt{n}/A\right)}
	\geq \left(\frac{A}{128K\sqrt{n}}\right)^{\frac{1}{2} N (x_j,r_j)}
	\geq \left(\frac{A}{128K\sqrt{n}}\right)^{\frac{N}{2}}.
\label{EQ62}
\end{align}
For $s>0$,
we denote by $s q_{j,0}$ the closed cube with the same center as $q_{j,0}$ and with the side-length equal to $s$ times that of $q_{j,0}$.
Note that
\begin{align}
	2q_{j,0} 
	\subset
	B\left( x_j, \frac{4\sqrt{n}}{A } \right).
	\label{EQ63}
\end{align}
From \eqref{EQ60} and \eqref{EQ62}--\eqref{EQ63}, it follows that
\begin{align}
	H({2q_{j,0}})
	\leq 
	H
	\left(
	x_j,  \frac{4\sqrt{n}}{A}
	\right)
	\leq \left(\frac{128K\sqrt{n}}{A}
	\right)^{\frac{N}{2}}
	H(x_j, 1/32 K)
	\leq \left(\frac{128K\sqrt{n}}{A}\right)^{\frac{N}{2}}
	M.
	\label{EQ64}
\end{align}
Using \eqref{EQ59}, \eqref{EQ64}, and the standard interior elliptic estimates, we obtain
\begin{align}
	\begin{split}
	\sup_{y\in \Lambda}
	&
	\left(|u(y)|^2
	+
	|\nabla u (y)|^2
	\right)
	\leq
	2(\Vert u\Vert_{L^\infty (q_{j,0})}^2
	+
	\Vert \nabla u\Vert_{L^\infty (q_{j,0})}^2)
	\leq
	C
	A^{n +2}
	\int_{2q_{j,0}} |u|^2
	\\&\quad
	\leq
	C
	A^{n+2}
	\left(\frac{128K\sqrt{n}}{A}\right)^{\frac{N}{2}} M
	\leq
	e^{-(N \log A )/4}
	M
	,
	\end{split}
	\label{EQ65}
\end{align}
where the last step follows by taking $N_0>0$ and $A_0>0$ sufficiently large depending on~$K$.

Let $q$ be a
cube in the upper half space with center $p=(0,0,\ldots, 1/32K\sqrt{n})$
and side-length $1/16K\sqrt{n}$,
with one face on the hyperplane $\{x_n=0\}$.
It is readily checked that
$q \subset B(0, 1/16K)$ and
$
B\left(0,1/32K\sqrt{n}
\right)
\cap
\{x_n=0\}
\subset
\partial q \cap \{x_n=0\}
\subset
\Lambda=B(0,1/16 K) \cap \{x_n = 0\}$.
Let $v(x)=u(x)/M^{1/2}$.
From \eqref{EQ65}, we infer
\begin{align}
	\sup_{y\in \Lambda}
	\left(|v(y)|^2
	+
	|\nabla v (y)|^2
	\right)
	=
	\frac{1}{M}
	\sup_{y\in \Lambda}
	\left(|u(y)|^2
	+
	|\nabla u (y)|^2
	\right)
	\leq
	\eps:=e^{- (N \log A)/4}
	\label{EQ66}
\end{align}
and
\begin{align}
\int_{B^+ (0,1/8K)} |v|^2
=
\frac{1}{M}
\int_{B^+ (0,1/8)}
|u|^2
\leq1.
\label{EQ67}
\end{align}
Note that $\eps\in(0,1)$ can be chosen sufficiently small by taking $A_0, N_0>0$ sufficiently large.
Therefore,
using Lemma~\ref{cauchyunique} and \eqref{EQ66}--\eqref{EQ67},
it follows that 
\begin{align}
\sup_{B^+ (0,1/32K)} |v|
\leq 
C\eps^{\beta}
,
\label{EQ68}
\end{align}
where $C\geq1$ and $\beta>0$ are constants depending on~$K$.

Since $q/4$ has side-length $1/64K\sqrt{n}$, it is readily checked that
\begin{align}
	B \left( p, \frac{1}{256K\ccc\sqrt{n}} \right) \subset 
	\frac{q}{4}
	\subset
	B^{+} (0,1/32K),
	\label{EQ69}
\end{align}
where $\ccc\geq 1$ is the constant from Lemma~\ref{Lmonotonenew}.
Let
\begin{align}
	\mu=1/7,\quad
	r=\frac{1}{256K\ccc \sqrt{n}},
	\quad
	t=\frac{-\log (1-\mu/2)}{12 r\ccc}.
	\llabel{EQ70}
\end{align}
From \eqref{EQ68} and \eqref{EQ69}, it follows that
\begin{align}
	H (p, r)
	=\int_{B(p,r)} |u|^2
	\leq 
	\int_{B^{+} (0,1/32K)} |u|^2
	\leq
	CM \eps^{2\beta}.
	\label{EQ71}
\end{align}
Now, we fix the value
\begin{align}
K=\frac{12 \ccc}{\log(14/13)}
\cdot
\frac{4\sqrt{n}+1}{32\sqrt{n}}\geq 1,
   \llabel{EQ72}
\end{align}
which implies that $B(0,1/8K) \subset B(p,tr)$.

Using \eqref{EQ58} and \eqref{EQ71}, we arrive at
\begin{align}
	\frac{H(p,tr)}{H(p, r )}
	\geq 
	\frac{H(0,1/8 K)}{H(p, r)} 
	\geq 
	\frac{1}{C\eps^{2\beta}}.
	\llabel{EQ73}
\end{align}
On the other hand,
from Lemma~\ref{L04}, it follows that
\begin{align}
\frac{H(p,tr)}{H(p, r)} \leq t^{(1-\mu)^{-2} N (p,tr)+C} 
\leq
e^{C N(p,tr)+C}.
\llabel{EQ74}
\end{align}
Combining the above two inequalities and using \eqref{EQ66}, we get
\begin{align}
	C N(p,tr) 
	\geq
	\frac{\beta  N \log A }{2} -C
	\geq 
	\frac{\beta N \log A}{3},
	\llabel{EQ75}
\end{align}
where the last inequality follows by taking $N_0>0$ sufficiently large. 
Therefore, we conclude that
\begin{align}
N(Q)
\geq 
N(p,tr)
>2N,
\label{EQ76}
\end{align}
by taking $A_0>0$ sufficiently large.

Now, we consider the case when $Q=[-R/2, R/2]^n$ where
$R\in (0,1)$.
Let
\begin{align}
\tilde{u}(x)=u(Rx),\quad
\tilde{W}(x) =R W(Rx),\quad
\tilde{V}(x) = R^2 V(Rx)
\llabel{EQ77}
\end{align}
be defined in a neighborhood of $\tilde{Q}:=[-1/2,1/2]^n$.
Then
\begin{align}
\Delta \tilde{u}
=
\tilde{W}\cdot
\nabla 
\tilde{u}
+
\tilde{V}\tilde{u},
	\llabel{EQ78}
	\end{align}
where the coefficients $\tilde{V}$ and $\tilde{W}$ satisfy 
\begin{align}
\Vert \tilde{V}\Vert_{W^{1,\infty }(\tilde{Q})}
+
\Vert \tilde{W}\Vert_{W^{1,\infty} (\tilde{Q})}
\leq 1.
	\llabel{EQ79}
\end{align}
Therefore, we proceed as in \eqref{EQ58}--\eqref{EQ76} and use a rescaling argument to conclude the proof of the lemma.
\end{proof}

\subsection{Number of subcubes with large doubling index}
The following lemma provides an upper bound on the number of subcubes near the hyperplane $\{x_n=0\}$ with large doubling index.
\cole
\begin{Lemma}
\label{L08}
Let $u$ be a solution of \eqref{EQ06}--\eqref{EQ07} and $A_0,N_0>0$ the constants from Lemma~\ref{L07}.
For any $\delta \in (0,1)$, there exists a constant $k_0\in\NN$, depending on $A_0$ and $\delta$,
such that if $ A = A_0^k$, where $k\geq k_0$, and $N (Q) \geq 2 N_0$, then
\begin{align}
\# \left\{ q_{j,0} : N(q_{j,0}) \geq \frac{N (Q)}{2}
\right\} < \delta A^{n-1}.
\llabel{EQ80}
\end{align}
\end{Lemma}
\colb

\begin{proof}
Let $N=N(Q)$.
First,
we divide the cube $Q$ into $A_0^n$ subcubes and denote the subcubes that intersect $\{x_n=0\}$ by $q_{j_1,0}$, where $j_1=1,2,\ldots, A_0^{n-1}$.
Using Lemma~\ref{L07}, we infer that there exists a subcube $ q_{j_1,0}$  with $N(q_{j_1,0}) < N/2$, and
denote by $M_1$ the number of subcubes $q_{j_1,0}$ with $N(q_{j_1,0} )  
\geq 
N/2$ after the first division.
It follows that
\begin{align}
M_1 
\leq 
A_0^{n-1}-1.
\label{EQ81}
\end{align}
Next, we further divide each $q_{j_1,0}$ into  $A_0^n$  subcubes of equal size, where $j_1=1,2,\ldots, A_0^{n-1}$.
Denote by $q_{j_1,j_2,0}$ the subcubes that intersect with $\{x_n=0\}$, where $j_1,j_2=1,2,\ldots,A_0^{n-1}$.
Denote by $M_2$ the number of subcubes $q_{j_1,j_2,0}$ with $N(q_{j_1,j_2,0}) \geq N/2$ after the second division.
Note that if  $N(q_{j_1,0})< N/2$ for some $j_1=1,\ldots,A_0^{n-1}$, 
then $N(q_{j_1,j_2,0}) <N/2$ for any $j_2=1,\ldots,A_0^{n-1}$.
Using Lemma~\ref{L07} and \eqref{EQ81}, we obtain
\begin{align}
M_2 
\leq M_1 \cdot 
(A_0^{n-1}-1)
\leq 
(A_0^{n-1}-1)^2
.
\llabel{EQ82}
\end{align}

Let $k\in\mathbb{N}$. 
We iterate the above partitioning, and denote by $M_k$ the number of subcubes $q_{j_1,\ldots,j_k,0}$ with $N(q_{j_1,\ldots,j_k,0})\geq N/2$ after the $k$-th division. 
It follows that
\begin{align}
M_k 
\leq
(A_0^{n-1} -1)^k 
=
A_0^{k(n-1)} 
\left(
1-A_0^{1-n}
\right)^k
.
\label{EQ83}
\end{align}
For $\delta\in (0,1)$, we fix an integer $k_0\in \NN$ such that $(1-A_0^{1-n})^{k_0}<\delta$.
Let $A=A_0^k$, where $k\geq k_0$.
Then from \eqref{EQ83} it follows that
\begin{align}
	\# \left\{ q_{j,0} : N(q_{j,0}) \geq \frac{N (Q)}{2}
	\right\}
	=M_k 
	< 
	\delta A_0^{k(n-1)}
	=\delta A^{n-1},
   \llabel{EQ84}
\end{align}
which completes the proof of the lemma.
\end{proof}

The following lemma provides an upper bound on the number of subcubes with large doubling index.
\cole
\begin{Lemma}
\label{L10}
Let $u$ be a solution of \eqref{EQ06}--\eqref{EQ07}.
There exist constants $\eta, N_0>0$ and an odd integer $A>0$ such that
if $N (Q) \geq N_0$, then
\begin{align}
\#
\left\{
q: N (q)>\frac{N (Q)}{1+\eta}
\right\}
<\frac{1}{2} A^{n-1}.
\llabel{EQ85}
\end{align}
\end{Lemma}
\colb

\begin{proof}
Let $A_0 \geq 3$ be a multiple of $4n+1$ and $N_0>0$ a constant, both to be determined below.
We partition $Q$ inductively  into subcubes of equal size.
At each step, we divide $Q$ into $A_0^n$ subcubes of equal size.
Let $j\in \mathbb{N}$. 
The cube $Q$ is partitioned into $A_0^{nj}$ subcubes in the $j$-th step, and we denote all the subcubes by
\begin{align}
\{Q_{i_1,\ldots,i_j}: i_1,i_2,\ldots,i_j=1,2,\ldots, A_0^n\}.
\llabel{EQ86}
\end{align}
It is clear that $\diam Q_{i_1,\ldots,i_j}=\diam  Q/A_0^j$.
Let $\eta\in(0,1/10]$ be a constant to be determined below.
We say that a subcube $q$ is good if $N(q)\leq N(Q)/(1+\eta)$ and is bad if $N(q)> N(Q)/(1+\eta)$.
Fix $q = Q_{i_1, \ldots, i_j}$. Denote by $q_i$ the subcubes of $q$ in the $(j+1)$-th step, where $i=1,2,\ldots,A_0^n$.

\textit{Claim 1}:
All the bad subcubes $q_i \subset q$ are contained in the $2\diam q_i$ neighborhood of some hyperplane for sufficiently large $j\in \mathbb{N}$ and sufficiently small $\eta>0$. 
For the case when $q$ is good, i.e., $N(q)\leq N(Q)/(1+\eta)$, it follows that all the subcubes $q_i\subset q$ satisfy $N(q_i)\leq N(q)\leq N(Q)/(1+\eta)$  and the claim holds trivially. 
Now, we consider the case when $q$ is bad.
Denote
\begin{align}
	F 
	= 
	\left\{ x \in q : N(x,r) > \frac{N(Q)}{1+\eta} \text{ for some } 0 < r \leq \diam q_i=\frac{\diam q}{A_0} \right\}.
	\llabel{EQ87}
\end{align}
For the case when the relative width $\tilde{\omega}_q (F):=\mywidth F/\diam q=0$, $F$ lies between two arbitrarily close hyperplanes.
Hence, Claim~1 is proven.
We consider the case when $\tilde{\omega}_q (F) > 0$. 
Denote $\omega_0=1/A_0$.

\textit{Claim 2}:
There exist constants $N_{\omega_0},j_0, c_0>0$ depending on $\omega_0$
such that if $N(Q)>N_{\omega_0}$, $j\geq j_0$ and $\eta\in (0,c_0]$, then
\begin{align}
\tilde{\omega}_q (F)
	= 
\frac{\mywidth F}{\diam q}
<\omega_0.
\label{EQ88}
\end{align}

We postpone the proof of Claim~2 to Appendix~\ref{AppSimplex2}.
We shall take $A_0$ to depend on the constant from Lemma~\ref{L07} only, and thus $N_{\omega_0}$ is a constant. 
Let the parameter $N_0\geq N_{\omega_0}$ be as in the statement of Lemma~\ref{L10}.
Let
$\eta=\min \{1/10,c_0\}$ and $j\geq j_0$.
Using \eqref{EQ88}, we deduce that there exists a hyperplane $\mathcal{P}$ such that for any $x\in F$, we have
\begin{align}
\text{dist} (x, \mathcal{P})
< 
w_0 \cdot \diam q
.
	\llabel{EQ89}
\end{align} 
Let $q_i\subset q$ be a bad subcube. Then $N(x, r) > N(Q)/(1+\eta)$ for some $x \in q_i$ and $r \in (0, \diam q_i]$,
which implies $x\in F$.
For any $y\in q_i$, we have
\begin{align}
\text{dist}(y, \mathcal{P}) \leq |y - x| + 
\text{dist}(x, \mathcal{P}) 
<
\diam q_i + \omega_0 \cdot \diam q
= 
2 \diam q_i,
\label{EQ90}
\end{align}
since $\diam q_i = \diam q / A_0=\omega_0 \cdot\diam  q$. 
Therefore, Claim~1 is proven.
For the case when the hyperplane $\mathcal{P}$ is not the same as the equator hyperplane of $q$, we may adapt the arguments in \cite[Theorem~5.1]{Lo1} by choosing a bigger cube containing $q$ that has $\mathcal{P}$ as its equator hyperplane.
We omit the details here and assume that $\mathcal{P}$ is the equator hyperplane of~$q$.

Recall that $q=Q_{i_1,\ldots,i_j}$ (good or bad) is divided into $A_0^n$ subcubes $q_i$, where $i=1,\ldots,A_0^n$ and $A_0$ is a multiple of~$4n+1$.
This can be viewed as dividing $q$ into $(A_0/(4n+1))^n$ subcubes $\tilde{q}_k\subset q$ with $\diam \tilde{q}_k = (4n+1) \cdot\diam q_i$, where $k=1,2,\ldots, (A_0/(4n+1))^n$. 
From \eqref{EQ90}, we infer that all the bad subcubes $q_i$ are contained in $\tilde{q}_{k,0}$,
where $\tilde{q}_{k,0}$ are the subcubes that intersect $\{x_n=0\}$.
Consequently,
\begin{align}
(4n+1)^n
\cdot
\# 
\left\{
\tilde{q}_{k,0}: N(\tilde{q}_{k,0})>\frac{N(Q)}{1+\eta}
\right\}
\geq 
\# 
\left\{
q_i: N(q_i)>\frac{N(Q)}{1+\eta}
\right\}.
\label{EQ91}
\end{align}

Towards a contradiction, we assume that 
\begin{align}
\# 
\left\{
q_i: N(q_i)>\frac{N(Q)}{1+\eta}
\right\}
\geq \frac{1}{2} A_0^{n-1}.
\label{EQ92}
\end{align}
Let $\delta= 1/2 (4n+1)^n\in (0,1/10)$
and $\constanta, \constantn>0$ be the constants from Lemma~\ref{L07}.
From Lemma~\ref{L08}, we deduce that there exists a constant $k_0\in\NN$ such that if $A_0/(4n+1)= \constanta^{k_0}$ and $N(q) \geq 2 \constantn$, then
\begin{align}
\# \left\{ \tilde{q}_{k,0} : N(\tilde{q}_{k,0}) 
\geq 
\frac{N (q)}{2}
\right\} < 
\delta
\left(
\frac{A_0}{4n+1}
\right)^{n-1}
<\frac{A_0^{n-1}}{2 (4n+1)^n}.
\label{EQ93}
\end{align}
Let $N_0 =\max\{N_{\omega_0}, 4\constantn\}$ and $A_0= (4n+1) \constanta^{k_0}$.
For the case when $q$ is a bad subcube, we have
\begin{align}
N(q) > \frac{N(Q)}{1+\eta} 
\geq 
\frac{4\constantn}{1+\eta}
\geq 
2\constantn,
\llabel{EQ94}
\end{align}
since $\eta\in (0,1/10]$.
Therefore, the assumptions needed to obtain \eqref{EQ93} are satisfied and
we conclude \eqref{EQ93} holds when $q$ is a bad subcube.
From \eqref{EQ91}--\eqref{EQ93}, we arrive at
\begin{align}
	\# 
\left\{
\tilde{q}_{k,0}: N(\tilde{q}_{k,0})>\frac{N(Q)}{1+\eta}
\right\}
>
\# \left\{ 
\tilde{q}_{k,0} : N(\tilde{q}_{k,0}) 
\geq 
\frac{N (q)}{2}
\right\},
   \llabel{EQ95}
\end{align}
from where
\begin{align}
\frac{	N(q)}{2} 
\geq \frac{ N(Q) }{1+\eta} \geq 
\frac{10 N(Q)}{11}.
\llabel{EQ96}
\end{align}
This contradicts with $N(q) \leq N(Q)$, and thus \eqref{EQ92} does not hold for the case when $q$ is a bad subcube.
Clearly, \eqref{EQ92} does not hold for the case when $q$ is a good subcube.
Consequently, we obtain
\begin{align}
\# 
\left\{
q_i: N(q_i)>\frac{N(Q)}{1+\eta}
\right\}
< \frac{1}{2} A_0^{n-1},
\label{EQ97}
\end{align}
for every $q$ after the $j_0$-th step.
Using \eqref{EQ97} and similar arguments as in \cite[Theorem~5.1]{Lo1}, we conclude the proof of the lemma.
\end{proof}

\subsection{Upper bound on the nodal set}
The following lemma provides an
upper bound on the nodal set in terms of the doubling index.
\cole
\begin{Lemma}
\label{L11}
Let $u$ be a solution of \eqref{EQ06}--\eqref{EQ07}.
There exist constants $\alpha, C>0$ such that for any cube $Q\subset \Omega$ with side-length $S(Q)\leq 1$, we have
\begin{align}
	\mathcal{H}^{n-1}
	(\{u=0\}\cap Q)
	\leq
	C\diam^{n-1} (Q)
	N^{\alpha} (Q).
	\llabel{EQ98}
\end{align}
\end{Lemma}
\colb

\begin{proof}
For $N>0$, we define the function
\begin{align}
F(N)
:=
\sup \frac{\mathcal{H}^{n-1} (\{u=0\}\cap Q)}{\diam^{n-1} Q},
\llabel{EQ99}
\end{align}
where the supremum is taken over all cubes $Q\subset \Omega$ with $S(Q)\leq 1$ and all solution $u$ of \eqref{EQ06}--\eqref{EQ07} with $N (Q)\leq N$.
From \cite[Theorem~1.7]{HS}, it follows that $F(N)<\infty$ for any
$N>0$. 

Let $\eta,N_0,A>0$ be the constants from Lemma~\ref{L10}.
We claim that
\begin{align}
F(N)\leq 4A  
\cdot
F 
\left(
\frac{N}{1+\eta}
\right).
\label{EQ100}
\end{align}
for all $N\geq (1+\eta) N_0$.
For the sake of contradiction, we assume that there exists some $N\geq (1+\eta) N_0$ such that
\begin{align}
F(N)>4A \cdot F
\left(
\frac{N}{1+\eta}
\right).
\label{EQ101}
\end{align}
By definition of $F(N)$,
there exists a cube $Q\subset \Omega$  and a solution $u$ to \eqref{EQ06}--\eqref{EQ07} with $S(Q)\leq 1$ and $N(Q) \leq N$ such that
\begin{align}
\frac{\mathcal{H}^{n-1} (\{u=0\} \cap Q)}{\diam^{n-1} Q}
>\frac{3}{4} F(N).
\label{EQ102}
\end{align}
We partition $Q$ into $A^n$ subcubes $q_j$ of equal size, where $j=1,2,\ldots, A^n$. 
Denote the two groups by
\begin{align}
G_1=
\left\{
q_j: \frac{N}{1+\eta} <N (q_j) \leq N
\right
\}
\llabel{EQ103}
\end{align}
and
\begin{align}
G_2
=
\left\{
q_j: N (q_j) \leq \frac{N}{1+\eta}
\right
\}.
   \llabel{EQ104}
\end{align}
It is clear that $|G_1|+|G_2|=A^n$.
For the case when $N(Q)\geq N/ (1+\eta)$, we infer $N(Q)\geq N_0$ and thus from Lemma~\ref{L10} it follows that $|G_1|\leq 
\frac{1}{2} A^{n-1}$.
For the case when $N(Q)<N/(1+\eta)$, we have $N(q_j)\leq N(Q)< N/(1+\eta)$ for all $j=1,2,\ldots,A^{n}$ which implies that $|G_1|\leq \frac{1}{2} A^{n-1}$.
Hence, we conclude that $|G_1|\leq \frac{1}{2} A^{n-1}$.

By subadditivity of the Hausdorff measure and \eqref{EQ101}, we have
\begin{align}
\begin{split}
\mathcal{H}^{n-1}
(\{u=0\} \cap Q)
&\leq
\sum_{q_j\in G_1}	\mathcal{H}^{n-1}
(\{u=0\} \cap q_j)
+
	\sum_{q_j\in G_2}	\mathcal{H}^{n-1}
(\{u=0\} \cap q_j)
\\&
\leq
|G_1| \cdot
F(N) \cdot \diam^{n-1}q_j
+
	|G_2| \cdot
F\left(\frac{N}{1+\eta}\right) 
\cdot \diam^{n-1} q_j
\\&
\leq
\frac{1}{2} A^{n-1}
F(N) \cdot
\left(\frac{\diam Q}{A}
\right)^{n-1}
+
A^n
\cdot
\frac{	F(N)}{4A} \cdot 	\left(\frac{\diam Q}{A}
\right)^{n-1}
\\&
\leq
\frac{3}{4} F(N) \cdot \diam^{n-1} Q,
\llabel{EQ105}
\end{split}
\end{align}
contradicting with~\eqref{EQ102}.
Therefore, \eqref{EQ100} holds for all $N\geq (1+\eta) N_0 $, and it follows that
\begin{align}
F(N)\leq C N^\alpha
\comma
N>0,
\llabel{EQ106}
\end{align} 
where $C,\alpha>0$ are constants.
The proof of the lemma is thus completed.
\end{proof}

\section{Upper bound on the doubling index}
\label{sec05}
We recall the following three-ball inequality from~\cite[Proposition~2]{Da}.
\cole
\begin{Lemma}[\cite{Da}]
	\label{EQ107}
Let $u$ be a solution of \eqref{EQ06}--\eqref{EQ07} with $\Omega=B_R (0)$.
For $0<r_1<r_2<\frac{3}{2} r_2<r_3\leq R$, we have
\begin{align}
	H (x,r_2)
	\leq
	\left(
	\frac{r_3}{r_2}
	\right)^{CR^2}
	H (x,r_1)^{\gamma}
	H (x,r_3)^{1-\gamma},
	\llabel{EQ108}
\end{align}
for any $x\in B_R$, where
\begin{align}
	\gamma
	=
	\frac{\log r_3- \log \left(
		\frac{3r_2}{2}
		\right)
	}{\log r_3
	+\left(\frac{5e}{3}-1\right)
	\log \left(
	\frac{3r_2}{2}
	\right)
	-\frac{5e}{3} \log r_1
	}.
\llabel{EQ109}
\end{align}
\end{Lemma}
\colb

We adapt the proof of the three-ball inequality from \cite[Proposition~2]{Da} to establish the following four-ball inequality.
\cole
\begin{Lemma}
	\label{Lfourball}
Let $u$ be a solution of \eqref{EQ06}--\eqref{EQ07} with $\Omega=B_R (0)$.
For any $0<\rho_1<2\rho_1< \rho_2 <2\rho_2 <R$, we 
have 
\begin{align}
\begin{split}
	&
	\log
	\left(
	\frac{H(x,2\rho_1)}{H(x,\rho_1)}
	\right)
	\leq
	C
	+
	C
	\log
	\left(
	\frac{H(x,2\rho_2)}{H(x,\rho_2)}
	\right)
	+
	CR^2,
	\llabel{EQ110}
\end{split}
\end{align}
for any $x\in B_{R}$.
\end{Lemma}
\colb
We emphasize that the constant $C$ is independent of~$R$.

\begin{proof}
	Let $x\in B_R$ and $r>0$. Let
	$\alpha\geq 2$.
	We define
	\begin{align}
		&\bhh (x,r)
		=
		\int_{B_r (x)}
		|u|^2 (r^2 - |y|^2)^{\alpha-1}\,dy,
		\label{EQ111}
		\\&
		\bdd (x,r)
		=\int_{B_r (x)} 
		|\nabla u|^2 
		(r^2 - |y|^2)^{\alpha}\,dy,
		\label{EQ112}
		\\&
		\bnn (x,r)=
		\frac{\bdd(x,r)}{\bhh (x,r)}
		\label{EQ113}
	\end{align}
(the weighted norms \eqref{EQ111}--\eqref{EQ112} and the frequency \eqref{EQ113} were introduced in \cite{Ku2}).
Using (4.19) and (4.20) in \cite[Proposition~2]{Da}, we infer that
	\begin{align}
		\frac{\bhh'(r)}{\bhh (r
			)}
		\geq
		\frac{2\alpha+n-2}{r}
		+
		\frac{3 \left(\bnn (r)+\frac{r^4 }{8\alpha}
			\right)}{4\alpha r}
		-
		\frac{3r^3}{32\alpha^2} 
		-
		\frac{2r}{\alpha} 
		\label{EQ114}
	\end{align}
	and
	\begin{align}
		\frac{\bhh'(r)}{\bhh (r
			)}
		\leq
		\frac{2\alpha+n-2}{r}
		+
		\frac{5 \left(\bnn (r)+\frac{r^4 }{8\alpha}
			\right)
		}{4\alpha r}
		+
		\frac{2 r}{\alpha}.
		\label{EQ115}
	\end{align}
Here and below, for simplicity of notation, we omit the $x$-dependence.
Note that \eqref{EQ114} and \eqref{EQ115} hold for any $\alpha\geq 2$.
	Let $\alpha= 2 R^2$.
Define 
\begin{align}
	\tilde{\bnn} (x,r)
	= 
	\left(\bnn (x,r) +\frac{r^4}{8\alpha}
	\right) 
	e^{\frac{r^2}{4\alpha}}.
	\llabel{EQ116}
\end{align}
 Using the monotonicity of $\tilde{\bnn} (x,r)$ in $r>0$ from \cite[Corollary~2]{Da}, we deduce that
	\begin{align}
		\bnn (r)
		+
		\frac{r^4}{8\alpha}
		=
		\tilde{\bnn} (r)
		e^{-r^2/4\alpha}
		\geq
		\tilde{\bnn} 
		(3\rho_2 /2)
		e^{-1}
		\comma
		r\in[3\rho_2/2,2\rho_2]
		\label{EQ117}
	\end{align}
	and
	\begin{align}
		\bnn (r)
		+
		\frac{r^4}{8\alpha}
		=\tilde{\bnn} (r) e^{-r^2/4\alpha}
		\leq 
		\tilde{\bnn} (3\rho_1)
		\comma
		r\in [\rho_1, 3 \rho_1].
				\label{EQ118}
	\end{align}
	We integrate \eqref{EQ114} in $r$ from $3\rho_2/2$ to $2\rho_2$ and appeal to \eqref{EQ117}, yielding
	\begin{align}
		\begin{split}
			&
			\log
			\left(
			\frac{\bhh (2\rho_2
				)}{\bhh(3 \rho_2/2)}
			\right)
			=
			\int_{3 \rho_2/2}^{2\rho_2}
			\frac{\bhh'(r)}{\bhh (r)}\,dr
			\geq
			\int_{3 \rho_2/2}^{\rho_2}
			\left(\frac{2\alpha+n-2}{r}
			+
			\frac{3 \left(\bnn (r)+\frac{r^4 }{8\alpha}\right)}{4\alpha r}
			-
			\frac{3 r^3}{32\alpha^2} 
			-
			\frac{2 r}{\alpha} \right)\,dr
			\\&\quad
			\geq
			\frac{3
				\left(
				2\alpha+n-2 + 
			\frac{	\tilde{\bnn} (3\rho_2/2)}{\alpha} \right)
		}{4 e}
			\log
			\left(
			\frac{2\rho_2}{3\rho_2/2}
			\right)
			-
		C
			.
			\llabel{EQ119}
		\end{split}
	\end{align}
	Similarly, we integrate \eqref{EQ115} in $r$ from $\rho_1$ to $3 \rho_1$ and use \eqref{EQ118} to get
\begin{align}
\begin{split}
	&\log
	\left(
	\frac{\bhh (3\rho_1
		)}{\bhh(\rho_1)}
	\right)
	=
	\int_{\rho_1}^{3\rho_1}
	\frac{\bhh'(r)}{\bhh (r)}\,dr
	\leq	\int_{\rho_1}^{3\rho_1}
	\left(\frac{2\alpha+n-2}{r}
	+
	\frac{5 \left(\bnn (r)
		+
		\frac{r^4}{8\alpha}
		\right)}{4\alpha r}
	+
	\frac{2r}{\alpha} \right)\,dr
	\\&\quad\quad
	\leq
	\frac{5}{4}
	\left(
	2\alpha+n-2 
	+ 
	\frac{\tilde{\bnn} (3\rho_1) }{\alpha}
	\right)
	\log
	\left(
	\frac{3\rho_1}{\rho_1}
	\right)
	+
	C
	.
			\llabel{EQ120}
\end{split}
\end{align}
Since $\tilde{\bnn} (r)$ is a non-decreasing function of $r$ and $\rho_2\geq 2\rho_1$, we infer that $\tilde{\bnn} (3\rho_2/2)\geq \tilde{\bnn} (3\rho_1)$.
Combining the above two inequalities, we arrive at
\begin{align}
\begin{split}
&
\frac{\log \left(
	\frac{\bhh (3 \rho_1)}{\bhh (\rho_1
		)}
	\right)
	-C	
}{\log 3
}
\leq
\frac{5}{4} 
\left(
2\alpha+n-2
+
\frac{\tilde{\bnn} (3 \rho_1)}{\alpha}
\right)
\leq
\frac{
	\log \left(
	\frac{\bhh (2\rho_2)}{\bhh (3 \rho_2/2)}
	\right)
	+C
}{\log (4/3)}
\cdot
\frac{5e}{3}
,
   \llabel{EQ121}
\end{split}
\end{align}
which leads to
\begin{align}
\begin{split}
\log
\left(
\frac{\bhh (3 \rho_1)}{\bhh (\rho_1)}
\right)	
\leq
C
+
C
\log
\left(
\frac{\bhh (2 \rho_2)}{\bhh (3 \rho_2/2)}
\right).
\label{EQ122}
\end{split}
\end{align}

For any $r\in (0,R/2]$,
it is easy to check that
\begin{align}
\bhh (x,r)
\leq
r^{2(\alpha-1)}
\int_{B_r (x)}
|u|^2
=r^{2(\alpha-1)}
H(x,r)
\label{EQ123}
\end{align}
and
\begin{align}
	\bhh (x,3 r/2)
	\geq
	\left(\frac{5r^2}{4} \right)^{\alpha-1}
	\int_{B_{r} (x)}
	|u|^2
	=
	\left(\frac{5}{4} \right)^{\alpha-1}
	r^{2(\alpha-1)}
	H(x,r).
	\label{EQ124}
\end{align}
Using \eqref{EQ122}--\eqref{EQ124}, we conclude that
\begin{align}
\begin{split}
	&
	\log
	\left(
	\frac{H(x,2\rho_1)}{H(x,\rho_1)}
	\right)
	\leq
	C
	+
	C
	\log
	\left(
	\frac{H(x,2\rho_2)}{H(x,\rho_2)}
	\right)
	+
	C R^2,
	\llabel{EQ125}
\end{split}
\end{align}
completing the proof of the lemma.
\end{proof}

The following lemma provides an upper bound on the doubling index.
\cole
\begin{Lemma}
	\label{Lnodal}
Let $u$ be a solution of \eqref{EQ06}--\eqref{EQ07} with $\Omega=B_R (0)$.
Suppose that there exists a constant $K\geq 1$ such that $\Vert u\Vert_{L^2 (B_R (0))} \leq  K\Vert u\Vert_{L^2 (B_{R/2} (0))}$.
There exists a constant $C\ge1$ such that
\begin{align}
	\log
	\left(
	\frac{H(x,2r)}{H(x,r)}
	\right)
	\leq
	C(R^2 +\log K),
	\llabel{EQ127}
\end{align}
for any $x\in B_{R/2}$ and $r\in (0, R/8]$.
\end{Lemma}
\colb
We emphasize that the constant $C$ is independent of $R$ and~$K$.

\begin{proof}
	Without loss of generality, we assume that $\Vert u\Vert_{L^2 (B_R (0))}=1$. 
	There exists a finite open cover
	\begin{align}
		B_{R/2} (0)
		\subset
		\bigcup_{i\in I}
		B_{R/4} (x_i),
		\llabel{EQ126}
	\end{align}
where $|I|\leq C$ and $x_i\in B_{R/2}$.
Hence, there exists a ball $B_{R/4} (x_0)$ such that
\begin{align}
	\Vert u\Vert_{L^2 (B_{R/4} (x_0))}
	\geq 
	\frac{\Vert u\Vert_{L^2 (B_{R/2})}}{C}
	\geq
	\frac{1}{CK},		\llabel{EQ128}
\end{align}
where $x_0\in B_{R/2}$.
Let $x\in B_{R/2}$. We may connect $x_0$ to $x$ by a chain of overlapping balls and invoke Lemma~\ref{EQ107} to obtain
\begin{align}
	\log
	\left(
	\frac{H(x,R/2)}{H(x,R/4)}
	\right)
	\leq
	C (R^2+\log K).
	\llabel{EQ129}
\end{align}
The lemma then follows by appealing to Lemma~\ref{Lfourball}.
\end{proof}

\begin{proof}[Proof of Theorem~\ref{T01}]
Let $w$ be a solution of \eqref{EQ03}--\eqref{EQ04} in a neighborhood of $\bar{B_2}  \subset \mathbb{R}^{n}$.
Let
\begin{align}
	u (x) 
	=
	w\left(
	\frac{x}{M}
	\right),
	\quad
	\tilde{W} (x)= 
	\frac{1}{M}
	W\left(
	\frac{x}{M}
	\right)
	,\quad
	\tilde{V}(x)
	= 
	\frac{1}{M^2}
	V\left( 
	\frac{x}{M}
	\right).
	\llabel{EQ130}
\end{align}
It is readily checked that $u$ satisfies
\begin{align}
	\Delta u
	=\tilde{W} \cdot \nabla u
	+
	\tilde{V} u
	\inin{B_{2M}},
	\llabel{EQ131}
\end{align}
where the coefficients $\tilde{W}$ and $\tilde{V}$ satisfy
\begin{align}
	\Vert \tilde{W}\Vert_{W^{1,\infty} (B_{2M})}
	+
	\Vert
	\tilde{V}\Vert_{W^{1,\infty} (B_{2M})}
	\leq 1.
   \llabel{EQ132}
\end{align}
From \eqref{EQ04}, it follows that
\begin{align}
	\Vert u\Vert_{L^2 (B_{2M})}
	\leq
	e^{\kappa}
	\Vert u\Vert_{L^2 (B_{M})}.
	\label{EQ133}
\end{align}
For any $y\in B_1$, we have
\begin{align}
	\begin{split}
		&
		\mathcal{H}^{n-1}
		(\{x\in B_{M^{-1}} (y): w(x)=0\}
		)
		=
		\mathcal{H}^{n-1}
		(\{x\in B_{M^{-1}} (y): u(Mx)=0\}
		)
		\\&\indeq
		=
		\mathcal{H}^{n-1}
		\left(
		M^{-1} \{x\in B_{1} (My): u (x)=0\}
		\right)
		\\&\indeq
		\leq
		M^{-(n-1)}
		\mathcal{H}^{n-1}
		\left(
		\{x\in B_1 (My): u (x)=0\}
		\right).
		\label{EQ134}
	\end{split}
\end{align}
Let $Q\subset B_{M}$ be a cube with $\diam Q =1/8$.
Using \eqref{EQ133}, as well as Lemmas~\ref{L11} and \ref{Lnodal}, we get
\begin{align}
	\mathcal{H}^{n-1}
	(\{u=0\}\cap Q)
	\leq
	C
	N_u^{\alpha} (Q)
	\leq
	C(M^2 +\kappa)^C.
	\llabel{EQ135}
\end{align}
Since we may cover $B_1$ by a finite number of cubes with diameter $1/8$, we deduce that
\begin{align}
	\mathcal{H}^{n-1}
	(
	\{x\in B_1 (My):u(x)=0\}
	)
	\leq
	C(M^2+\kappa)^C
	\comma
	y\in B_1.
	\label{EQ136}
\end{align}
There exists an open cover
\begin{align}
	B_1\subset
	\bigcup_{j\in J}
	B_{M^{-1}} (y_j),
	\label{EQ137}
\end{align}
where $|J|\leq CM^n$ and $y_j\in B_1$.
Using the subadditivity and \eqref{EQ134}--\eqref{EQ137}, we arrive at
\begin{align}
\begin{split}
	&
		\mathcal{H}^{n-1}
	(\{x\in B_1:w(x)=0\})
	\leq
	\sum_{j\in J}
	\mathcal{H}^{n-1}
	(\{x\in B_{M^{-1}} (y_j): w(x)=0\})
	\\&\quad
	\leq
	M^{-(n-1)}
	\sum_{j\in J}
	\mathcal{H}^{n-1}
	(\{x\in B_1 (My_j):u(x)=0\})
	\\&\quad
	\leq
	CM 
	(M^2+\kappa)^C,
	\llabel{EQ138}
\end{split}
\end{align}
completing the proof of the theorem.
\end{proof}

\appendix
\section{Auxiliary lemmas}
\subsection{Proof of Lemma~\ref{LSimplex}}
\label{AppSimplex}
We recall the following Euclidean geometry lemma from \cite[Section~2, page 225]{Lo1}:
for $\alpha>0$, 
there exist constants $C_{\alpha}>0$ and $K_{\alpha}
\geq 2/\alpha$ such that 
\begin{align}
	B(x_0, (1+C_\alpha) K_\alpha \diam S)
	\subset
	\bigcup_{j=1}^{n+1} B(x_j, K_\alpha \diam S).
	\label{EQ150}
\end{align}
We emphasize that the constants $C_\alpha, K_\alpha>0$ are independent of the simplex~$S$.
Set $r=K_\alpha\diam S$
and $M=\max_{j= 1,2,\ldots,n+1} H(x_j,r)$.
Then $M=H(x_i,r)$
for some $i\in \{1,2,\ldots, n+1\}$. 
From~\eqref{EQ150}, it follows that
\begin{align}
	H(x_0, (1+C_\alpha) r)
	\leq 
	\sum_{j=1}^{n+1}
	H(x_j, r)
	\leq
	(n+1)
	M.
	\label{EQ151}
\end{align}
Let $t\in [2 (1+C_\alpha),-\log (1-\eps/2)/24\ccc r]$ and $\eps\in (0,1/10]$ be constants to be determined below, where $\ccc\geq 1$ is the constant from Lemma~\ref{Lmonotonenew}; in particular, we shall take $\eps$ to depend on $t$ and $C_\alpha$.
Let
\begin{align}
	C_0
	\in 
	\left(0,\min 
	\left\{\frac{-\log (1-\eps/2)}{24}, \frac{-\log (1-\eps/2)}{48 \ccc K_\alpha
		 (1+C_\alpha)} 
	 \right\}
	 \right).
	 \llabel{EQ200}
\end{align}
%
%
Since $r=K_\alpha \diam S < K_\alpha C_0$, we have
\begin{align}
	\frac{	-\log (1-\eps/2)}{\ccc r (1+C_\alpha)}
	\geq 
	\frac{-\log (1-\eps/2)}{\ccc K_\alpha
	C_0 (1+C_\alpha)} 
	> 
	48.
	\label{EQ152}
\end{align}
Note that $r_j \leq K_\alpha \diam S/2=r/2$, where $j=1,\ldots,n+1$.
Denote $N=\min_{j=1,\ldots,n+1}N(x_j,r_j)$.
From \eqref{EQ152} and Lemma~\ref{L02}, it follows that 
\begin{align}
	N(x_i, r) 
	\geq 
	(1-\eps)^3 N(x_i,r_i)-C
	\geq
	(1-\eps)^3 N-C.
	\llabel{EQ153}
\end{align}
Hence, from Lemma~\ref{L04}, we obtain
\begin{align}
	H(x_i, tr)
	\geq	
	t^{(1-\epsilon)^2 N(x_i, r) -C}
	H(x_i,r)
	\geq
	t^{(1-\epsilon)^5 N -C}
	H(x_i,r)
	\geq 
	M t^{(1-\epsilon)^6 N },
	\label{EQ154}
\end{align}
where the last step follows
by taking $N_\alpha>0$ sufficiently large depending on~$\eps$.
Set $\delta=1/K_\alpha t =\diam S /t r$.
By the triangle inequality, we have
\begin{align}
	\bigcup_{j=1}^{n+1}
	B(x_j, tr)
	\subset 
	B\left(x_0,  \left(1+\delta
	\right) 
	tr\right),
   \llabel{EQ155}
\end{align}
which leads to
\begin{align}
	H(x_i, tr)
	\leq
	H(x_0, (1+\delta) tr)
	.
	\label{EQ156}
\end{align}
Note that \eqref{EQ152} implies $(1+C_\alpha) r\in (0, -\log (1-\eps/2) /24 \ccc)$.
It is readily checked that
\begin{align}
	(1+\delta)tr
	=tr+\diam S
	<
	-
	\frac{\log (1-\eps/2)}{24\ccc}
	-
	\frac{\log (1-\eps/2)}{24\ccc}
	=
	\frac{-\log (1-\eps/2)}{12\ccc}
   \llabel{EQ157}
\end{align}
and
\begin{align} 	
	 \frac{(1+\delta)t}{(1+C_\alpha)}
	\geq 2.
   \llabel{EQ158}
\end{align}
Using Lemma~\ref{L04}, we arrive at 
\begin{align}
	\begin{split}
			\log\left(
		\frac{H(x_0, (1+\delta) tr)}{H(x_0, (1+C_\alpha)r)}
		\right)
		\leq
	\left(	(1-\eps)^{-2}
		N(x_0, (1+\delta)tr)
		+
		C
		\right)
		\log
		\left(
		\frac{(1+\delta)t}{1+C_\alpha}
		\right)
		.
		\label{EQ159}
	\end{split}
\end{align}
From \eqref{EQ151} and \eqref{EQ154}--\eqref{EQ156}, it follows that
\begin{align}
	\begin{split}
		&
		\log\left(
		\frac{H(x_0, (1+\delta) tr)}{H(x_0, (1+C_\alpha)r)}
		\right)
		\geq
		\log
		\left(
		\frac{H(x_i,tr)}{H(x_0, (1+C_\alpha)r)}
		\right)
		\geq
		\log
		\left(
		\frac{M t^{(1-\eps)^6 N}}{(n+1)M}
		\right)
		\geq 
		(1-\eps)^7 N \log t,
		\label{EQ160}
	\end{split}
\end{align}
where the last step follows 
by taking $N_\alpha>0$ sufficiently large depending on $\eps$.

Let $t =\max\{2/K_\alpha C_\alpha, 2C_\alpha+2\}$.
Note that $t\in [2(1+C_\alpha), -\log (1-\eps)/24 \ccc K_\alpha \diam S]$ by taking $C_0>0$ sufficiently small.
Then $\delta=1/K_\alpha t\leq C_\alpha/2$.
Combining \eqref{EQ159} and \eqref{EQ160}, we obtain
\begin{align}
	\begin{split}
	&
	N(x_0, (1+\delta) tr) 
	\geq
	\frac{(1-\eps)^9 N \log t}{\log \left(
		\frac{(1+\delta)t}{1+C_\alpha}
		\right)}
	-C (1-\eps)^2
	\geq 
	\frac{(1-\eps)^{9}
	N}{1- \log \left(
	\frac{1+C_\alpha}{1+\delta}
	\right)/\log t}
	-C
	\\&\quad
	\geq
	\frac{(1-\eps)^{9} N}{1- \log \left(
	\frac{1+C_\alpha}{1+C_\alpha/2}
	\right)/\log t}
	-C.
	\llabel{EQ161}
	\end{split}
\end{align}
Let $N_\alpha>0$ be sufficiently large and $\eps \in (0,1/10]$ be sufficiently small depending on $C_\alpha$ and $t$, we obtain
\begin{align}
	N(x_0, (1+\delta)tr)
	> (1+C_1) N,
   \llabel{EQ162}
\end{align}
where $C_1>0$ is a constant.
Let $C_2=1+K_\alpha t>0$.
It follows that
\begin{align}
	N(x_0, C_2 \diam S)
	=
	N(x_0, (1+\delta) tr)
	> 
	(1+C_1) N,
	\llabel{EQ163}
\end{align}
completing the proof of the lemma.
\hfill $\square$

\subsection{Proof of Claim~2 in Lemma~\ref{L10}}
\label{AppSimplex2}
To obtain a contradiction, we
assume that $\tilde{\omega}_q (F)\geq \omega_0$. From Lemma~\ref{LG01}, it follows that there exists a constant $a=a(\omega_0)>0$ and a simplex $S$ with vertices $x_1,x_2,\ldots, x_{n+1}\in F$ such that
$\omega (S)>a$ and $\diam S>a\cdot \diam q$.
Hence,
for each $j=1,2,\ldots,n+1$, there exist $r_j\in (0, \diam q/A_0]$ such that
\begin{align}
	N(x_j, r_j)
	>
	\frac{N(Q)}{1+\eta}.
	\label{EQ164}
\end{align}
For this $a>0$, Lemma~\ref{LSimplex} implies that there exist constants $C_0, C_1, C_2, N_a>0$ and $K_a\geq 2/a$ depending on $a$ such that if
\begin{enumerate}
	\item $\omega (S) >a$ and $\diam S<C_0$,
	\item there exist $r_j\in (0,K_a\diam S/2]$, where $j=1,2,\ldots,n+1$, with
	\begin{align}
		\min_{j=1,\ldots,n+1} N(x_j, r_j) 
		\geq 
		N_a,
		\label{EQ165}
	\end{align}
\end{enumerate}
then
\begin{align}
	N (x_0, C_2 \diam S) 
	>(1+C_1) \min_{j=1,\ldots,n+1} N(x_j, r_j).
	\label{EQ166}
\end{align}
In fact, using a slight modification of the proof as in \cite[Lemma~4.1]{Lo1}, we can prove that
for any constant $N>0$, if the assumption \eqref{EQ165} is replaced by
\begin{align}
	\min_{j=1,\ldots,n+1}N(x_j,r_j)\geq 
	N
	\geq 
	N_a,
	\llabel{EQ167}
\end{align}
then the conclusion \eqref{EQ166} can be replaced by
\begin{align}
	N(x_0, C_2 \diam S)
	> 
	(1+C_1)N.
	\label{EQ168}
\end{align}
Let $N_{\omega_0}=  2 N_a$ and $c_0 =\min \{1,C_1\}$.
Denote $N:=N(Q)/(1+\eta)$.
Using \eqref{EQ164}, we obtain
\begin{align}
	N(x_j, r_j)
	\geq
	N
	>
	\frac{N(Q)}{2}
	\geq 
	\frac{2 N_a}{2}
	=
	N_a
	\comma
	j=1,2,\ldots,n+1.
		\label{EQ169}
\end{align}
It is clear that
\begin{align}
	r_j \leq 
	\diam q < \frac{\diam S }{a} 
	\leq 
	\frac{K_a  \diam S}{2}.
		\label{EQ170}
\end{align}
We choose $j_0>0$  sufficiently large depending on $C_0$ so that $3^{j_0} >\sqrt{n}/C_0$. 
It follows that for any $j\geq j_0$, we have
\begin{align}
	\diam S
	\leq 
	\diam q
	\leq 
	\frac{\sqrt{n}}{A_0^{j_0}}
	\leq
	\frac{\sqrt{n}}{3^{j_0}}<
	C_0.
	\label{EQ171}
\end{align}
From \eqref{EQ169}--\eqref{EQ171}, we infer that the assumptions of Lemma~\ref{LSimplex} are satisfied.
Therefore, we use \eqref{EQ168} to conclude that
\begin{align}
	N(x_0, C_2 \diam S) >
	(1+C_1) N
	=
	(1+C_1)\cdot \frac{N (Q)}{1+\eta} 
	.
   \llabel{EQ172}
\end{align}
We take $j_0 >0$ sufficiently large so that $3^{j_0 }> C_2$. 
Then, for any $j\geq j_0$, we have
\begin{align}
	C_2 \diam S
	\leq 
	C_2 \diam q
	= 
 \frac{C_2 \diam Q}{A_0^j} 
	\leq  \frac{C_2 \diam Q}{3^{j_0}}
	\leq
	\diam  Q,
   \llabel{EQ173}
\end{align}	
which leads to
\begin{align}
	N(Q)\geq N(x_0, C_2\diam S)
	>
	\frac{(1+C_1) N(Q)}{1+\eta}
	>
	N(Q),
	\llabel{EQ174}
\end{align}
where the last step follows since $\eta\in (0,c_0)$. 
Hence, we arrive at a contradiction.
\hfill
$\square$

\section*{Acknowledgments}
\rm
IK was supported in part by the NSF grant DMS-2205493.

\ifnum\sketches=1

\fi


\begin{thebibliography}{[ARRV]}



\bibitem[Al]{Al}
F.~J. Almgren Jr., 
\emph{Dirichlet's problem for multiple valued functions and the regularity of mass minimizing integral currents}, in {\it Minimal submanifolds and geodesics (Proc. Japan-United States Sem., Tokyo, 1977)}, pp. 1--6, North-Holland, Amsterdam-New York.


	
\bibitem[Ar]{Ar}
N.~Aronszajn, 
\emph{A unique continuation theorem for solutions of elliptic partial differential equations or inequalities of second order},
J. Math. Pures Appl.~(9)~\textbf{36} (1957), 235–249.



\bibitem[ABG]{ABG}
W.~O. Amrein, A.~M. Boutet~de~Monvel and V. Georgescu, \emph{$L\sp{p}$-inequalities for the Laplacian and unique continuation}, 
Ann. Inst. Fourier (Grenoble)~{\bf 31} (1981), no.~3, {\rm vii}, 153--168.


\bibitem[AKS]{AKS}
N. Aronszajn, A. Krzywicki and J. Szarski,
\emph{A unique continuation theorem for exterior differential forms on Riemannian manifolds},
Ark. Mat.~{\bf 4} (1962), 417--453 (1962).




\bibitem[ARRV]{ARRV}
G.~Alessandrini, L.~Rondi, E.~Rosset, and S.~Vessella, 
\emph{The stability for the Cauchy problem for elliptic equations}, 
Inverse Problems~{\bf 25} (2009), no.~12, 123004, 47 pp.


\bibitem[BG]{BG}
A. Banerjee and N. Garofalo, \emph{Quantitative uniqueness for elliptic equations at the boundary of $C^{1,{\rm Dini}}$ domains},
J. Differential Equations~{\bf 261} (2016), no.~12, 6718--6757.


\bibitem[Car]{Car}
T.~Carleman,
\emph{Sur un probl\`eme d'unicit\'e{} pur les syst\`emes d'\'equations aux d\'eriv\'ees partielles \`a{} deux variables ind\'ependantes},
Ark. Mat. Astr. Fys.~{\bf 26} (1939), no.~17, 9 pp.



\bibitem[Da]{Da}
B.~Davey,
\emph{A frequency function approach to quantitative unique continuation for elliptic equations}, 
arxiv 2506.19130.


\bibitem[Do]{Do}
R.~T.~Dong,
\emph{Nodal sets of eigenfunctions on Riemann surfaces}, 
J.~Differ.~Geom.~\textbf{36}, 493--506 (1992).


\bibitem[DF1]{DF1} 
H.~Donnelly and C.~Fefferman,
\emph{Nodal sets of eigenfunctions on {R}iemannian manifolds},
Invent. Math.~\textbf{93} (1988), no.~1, 161--183.


\bibitem[DF2]{DF2} 
H.~Donnelly and C.~Fefferman, \emph{Nodal sets for eigenfunctions of the Laplacian on surfaces},
J.~Amer. Math. Soc.~\textbf{3} (1990), no.~2,
333--353.




\bibitem[DZ]{DZ}
B.~Davey and J.~Zhu, 
\emph{Quantitative uniqueness of solutions to second-order elliptic equations with singular lower order terms}, 
Comm. Partial Differential Equations~{\bf 44} (2019), no.~11, 1217--1251.


\bibitem[GL]{GL}
N. Garofalo and F. Lin, \emph{Monotonicity properties of variational integrals, $A_p$ weights and unique continuation}, 
Indiana Univ. Math. J.~{\bf 35} (1986), no.~2, 245--268.	


\bibitem[HS]{HS}
R.~Hardt and L.~Simon,
\emph{Nodal sets for solutions of elliptic equations},
J. Differential Geom.~{\bf 30} (1989), no.~2, 505--522.


\bibitem[JK]{JK}
D.~S. Jerison and C.~E. Kenig, 
\emph{Unique continuation and absence of positive eigenvalues for Schr\"odinger operators}, 
Ann. of Math. (2)~{\bf 121} (1985), no.~3, 463--494.


\bibitem[Ku1]{Ku1}
I.~Kukavica,
\emph{Quantitative uniqueness for second-order elliptic operators}, 
Duke Math. J.~{\bf 91} (1998), no.~2, 225--240


\bibitem[Ku2]{Ku2}
I.~Kukavica,
\emph{Quantitative, uniqueness, and vortex degree estimates for solutions of the Ginzburg-Landau equation}, Electronic Journal of Differential Equations,      Vol.~\textbf{2000}(2000), No. 61, pp. 1--15.

\bibitem[KN]{KN}
I.~Kukavica and K.~Nystr\"{o}m,
\emph{Unique continuation on the boundary for {D}ini domains},
Proc. Amer. Math. Soc.~{\bf 126} (1998), 441--446.

\bibitem[KZZ]{KZZ}
C.~E. Kenig, J. Zhu and J. Zhuge, 
\emph{Doubling inequalities and nodal sets in periodic elliptic homogenization}, 
Comm. Partial Differential Equations~{\bf 47} (2022), no.~3, 549--584.


\bibitem[Lin]{Lin}
F.~Lin,
\emph{Nodal sets of solutions of elliptic and parabolic equations}, 
Comm. Pure Appl. Math.~{\bf 44} (1991), no.~3, 287--308.



\bibitem[Lo1]{Lo1} 
A.~Logunov, \emph{Nodal sets of {L}aplace eigenfunctions: polynomial
	upper estimates of the {H}ausdorff measure}, Ann. of Math. (2)~\textbf{187}
(2018), no.~1, 221--239.

\bibitem[Lo2]{Lo2} 
A.~Logunov,
\emph{Nodal sets of Laplace eigenfunctions: proof of {N}adirashvili's conjecture and of the lower bound in Yau's conjecture},
Ann. of Math. (2)~\textbf{187} (2018), no.~1, 241--262.

\bibitem[LoM1]{LoM1} 
A.~Logunov and E.~Malinnikova,
\emph{Nodal sets of Laplace eigenfunctions: estimates of the Hausdorff measure in dimensions two and three},
in {\it 50 years with Hardy spaces}, 333--344, Oper. Theory Adv. Appl., 261, Birkh\"auser/Springer, Cham.


\bibitem[LoM2]{LoM2} 
A.~Logunov and E.~Malinnikova,
\emph{Review of {Y}au's conjecture on zero sets of {L}aplace
	eigenfunctions}.
\emph{Current developments in mathematics 2018}, 179--212. Int. Press, Somerville, MA, 2020. 


\bibitem[LMNN]{LMNN}
A. Logunov, E.~Malinnikova, N.~Nadirashvili, and F.~Nazarov,
\emph{The sharp upper bound for the area of the nodal sets of Dirichlet Laplace eigenfunctions}, 
Geom. Funct. Anal.~{\bf 31} (2021), no.~5, 1219--1244.


\bibitem[LS]{LS}
F. Lin and Z. Shen, 
\emph{Nodal sets and doubling conditions in elliptic homogenization}, 
Acta Math. Sin. (Engl. Ser.)~{\bf 35} (2019), no.~6, 815--831.



\bibitem[LTY]{LTY}
F. Liu, L. Tian and X.~P. Yang,
\emph{Measure upper bounds of nodal sets of Robin eigenfunctions}, Math. Z.~{\bf 306} (2024), no.~1, Paper No. 14, 14 pp.



\bibitem[LZ]{LZ} 
F.~Lin and J.~Zhu,
\emph{Upper bounds of nodal sets for eigenfunctions of eigenvalue problems},
Math. Ann.~{\bf 382} (2022), no.~3-4, 1957--1984.


\bibitem[NV]{NV}
A. Naber and D. Valtorta, 
\emph{Volume estimates on the critical sets of solutions to elliptic PDEs}, 
Comm. Pure Appl. Math.~{\bf 70} (2017), no.~10, 1835--1897.


\bibitem[SS]{SS}
M.~Schechter and B.~Simon,
\emph{Unique continuation for Schr\"odinger operators with unbounded potentials},
J. Math. Anal. Appl.~{\bf 77} (1980), no.~2, 482--492.


\bibitem[Y]{Y} 
S.T.~Yau (ed.),
\emph{Seminar on {D}ifferential {G}eometry},
Annals of Mathematics Studies, vol. No. 102, Princeton University Press, Princeton, NJ;
University of Tokyo Press, Tokyo, 1982, Papers presented at seminars held during the academic year 1979--1980.


\bibitem[Z1]{Z1}
J. Zhu, 
\emph{Quantitative uniqueness of elliptic equations}, 
Amer. J. Math. {\bf 138} (2016), no.~3, 733--762.


\bibitem[Z2]{Z2}
J. Zhu, 
\emph{Doubling inequality and nodal sets for solutions of bi-Laplace equations}, 
Arch. Ration. Mech. Anal.~{\bf 232} (2019), no.~3, 1543--1595.


\bibitem[ZZ]{ZZ}
J.~Zhu and J.~Zhuge,
\emph{Nodal sets of Dirichlet eigenfunctions in quasiconvex Lipschitz domains},
arxiv 2303.02046.
\end{thebibliography}
\end{document}